\documentclass{amsart}
\usepackage{amssymb,setspace,diagrams}
\usepackage[l2tabu,orthodox]{nag}
\usepackage[headings]{fullpage}
\usepackage[colorlinks]{hyperref}
\newcommand{\QQ}{\mathbb Q}
\renewcommand{\AA}{\mathbb A}
\newcommand{\ZZ}{\mathbb Z}
\newcommand{\NN}{\mathbb N}
\newcommand{\BB}{\mathbb B}
\newcommand{\LL}{\mathcal{L}}

\newcommand{\calH}{\mathcal{H}}

\newcommand{\calG}{\mathcal{G}}
\newcommand{\calN}{\mathcal{N}}
\newcommand{\calF}{\mathcal{F}}
\newcommand{\calL}{\mathcal{L}}

\newcommand{\calO}{\mathcal{O}}
\newcommand{\J}{\mathfrak{J}}
\newcommand{\DD}{\mathbb D}
\DeclareMathOperator{\FFil}{\underline{\mathrm{Fil}}}

\newcommand{\into}{\hookrightarrow}
\newcommand{\vp}{\varphi}

\newcommand{\Qp}{\QQ_p}
\newcommand{\Zp}{\ZZ_p}
\newcommand{\Qpn}{\mathbb{Q}_{p,n}}
\newcommand{\Dcris}{\DD_{\mathrm{cris}}}

\newcommand{\Brig}{\BB_{\mathrm{rig},\Qp}^+}
\newcommand{\EA}{\calO_E\otimes\AA_{\Qp}^+}
\newcommand{\HIw}{H^1_{\mathrm{Iw}}}
\newcommand{\Bdag}{\BB_{\mathrm{rig},\Qp}^\dag}
\newcommand{\Nrig}{\NN_{\mathrm{rig}}}
\newcommand{\DdR}{\DD_{\mathrm{dR}}}
\newcommand{\DSen}{\DD_{\mathrm{Sen}}}

\DeclareMathOperator{\uCol}{\underline{\mathrm{Col}}}
\DeclareMathOperator{\Char}{Char}
\DeclareMathOperator{\Tw}{Tw}
\DeclareMathOperator{\Iw}{Iw}
\DeclareMathOperator{\Gal}{Gal}

\DeclareMathOperator{\image}{Im}

\DeclareMathOperator{\cris}{cris}
\DeclareMathOperator{\Fil}{Fil}
\DeclareMathOperator{\GL}{GL}
\DeclareMathOperator{\Sel}{Sel}

\DeclareMathOperator{\pr}{pr}

\newtheorem{theorem}{Theorem}[section]
\newtheorem{proposition}[theorem]{Proposition}
\newtheorem{lemma}[theorem]{Lemma}
\newtheorem{corollary}[theorem]{Corollary}
\newtheorem{definition}[theorem]{Definition}

\newtheorem{conjecture}[theorem]{Conjecture}
\newtheorem{assumption}[theorem]{Assumption}

\theoremstyle{definition} 
\newtheorem{remark}[theorem]{Remark}

\newtheorem{lettertheorem}{Theorem}

\begin{document}

\title{Coleman maps and the $p$-adic regulator}

\date{17th October 2010}

\author{Antonio Lei}
\address[Lei]{School of Mathematical Sciences\\
Monash University\\
Clayton, VIC 3800\\
Australia}
\email{antonio.lei@monash.edu}

\author{David Loeffler}
\address[Loeffler]{Mathematics Institute\\
Zeeman Building\\
University of Warwick\\
Coventry CV4 7AL, UK}
\email{d.a.loeffler@warwick.ac.uk}

\author{Sarah Livia Zerbes}
\address[Zerbes]{Department of Mathematics\\
  Harrison Building\\
  University of Exeter\\
  Exeter EX4 4QF, UK
}
\email{s.zerbes@exeter.ac.uk}
\thanks{The authors' research is supported by the following grants: ARC grant DP1092496 (Lei); EPSRC postdoctoral fellowship EP/F04304X/1 (Loeffler); EPSRC postdoctoral fellowship EP/F043007/1 (Zerbes).}

\begin{abstract}
In this paper, we study the Coleman maps for a crystalline representation $V$ with non-negative Hodge-Tate weights via Perrin-Riou's $p$-adic regulator $\calL_V$. Denote by $\calH(\Gamma)$ the algebra of $\Qp$-valued distributions on $\Gamma=\Gal(\Qp(\mu_{p^\infty})\slash\Qp)$. Our first result determines the $\calH(\Gamma)$-elementary divisors of the quotient of $\Dcris(V)\otimes(\Brig)^{\psi=0}$ by the $\calH(\Gamma)$-submodule generated by $(\vp^*\NN(V))^{\psi = 0}$, where $\NN(V)$ is the Wach module of $V$. By comparing the determinant of this map with that of $\calL_V$ (which can be computed via Perrin-Riou's explicit reciprocity law), we obtain a precise description of the images of the Coleman maps. In the case when $V$ arises from a modular form, we get some stronger results about the integral Coleman maps, and we can remove many technical assumptions that were required in our previous work in order to reformulate Kato's main conjecture in terms of cotorsion Selmer groups and bounded $p$-adic $L$-functions.
\end{abstract}


\maketitle

\tableofcontents


\section{Introduction}

 \subsection{Background}

  Let $p$ be an odd prime, and write $\QQ_\infty = \QQ(\mu_{p^\infty})$. Define the Galois groups $\Gamma=\Gal(\QQ_\infty/\QQ)$ and $\Gamma_1=\Gal(\QQ_\infty/\QQ(\mu_p))$. Note that $\Gamma\cong\Delta\times\Gamma_1$, where $\Delta$ is cyclic of order $p-1$ and $\Gamma_1 \cong \ZZ_p$. For $H\in\{ \Gamma,\Gamma_1\}$, denote by $\Lambda(H)$ the Iwasawa algebra of $H$, and $\Lambda_{\Qp}(H) = \Lambda(H) \otimes_{\Zp} \Qp$.
  
  Let $V$ be a crystalline representation of $\calG_{\QQ_p}$ of dimension $d$ with non-negative Hodge-Tate weights. We define 
  \[ H^1_{\Iw}(\Qp,V):=\Qp \otimes_{\Zp} \varprojlim_{n} H^1(\QQ(\mu_{p^n}),T),\]
  where $T$ is a $\calG_{\Qp}$-stable $\Zp$-lattice in $V$. This is a $\Lambda_{\Qp}(\Gamma)$-module independent of the choice of $T$. In \cite{leiloefflerzerbes10}, we construct $\Lambda(\Gamma)$-homomorphisms (called the Coleman maps)
  \[
   \uCol_i:H^1_{\Iw}(\Qp,V) \rTo \Lambda_{\Qp}(\Gamma)
  \]
  for $i=1,\dots,d$, depending on a choice of basis of the Wach module $\NN(V)$. In the case when $V=V_f(k-1)$, where $f=\sum a_n q^n$ is a modular eigenform of weight $k$ and level coprime to $p$ (we assume that $a_n\in\QQ$ for the time being in order to simplify notation) and $V_f$ is the $2$-dimensional $p$-adic representation associated to $f$, these maps have two important applications. Firstly, we can define two $p$-adic $L$-functions $L_{p,1}, L_{p,2}\in\Lambda_{\Qp}(\Gamma)$ on applying the Coleman maps to the localisation of the Kato zeta element as constructed in \cite{kato04}. In the supersingular case, i.e. when $p \mid a_p$, this enables us to obtain a decomposition of the $p$-adic $L$-functions defined in \cite{amicevelu75}, which are not elements of $\Lambda_{\Qp}(\Gamma)$ but of the distribution algebra $\calH(\Gamma)$. More precisely, we show that there exists a $2\times2$ matrix $\mathcal{M}\in M(2,\calH(\Gamma_1))$ depending only on $k$ and $a_p$ such that
  \[
   \begin{pmatrix} L_{p,\alpha}\\ L_{p,\beta} \end{pmatrix} = \mathcal{M} \begin{pmatrix} L_{p,1}\\ L_{p,2}\end{pmatrix}.
  \]
  This generalises the results of of Pollack \cite{pollack03} (when $a_p=0$) and Sprung \cite{sprung09} (when $f$ corresponds to an elliptic curve over $\QQ$ and $p=3$). Secondly, by modifying the local conditions at $p$ in the definition of the $p$-Selmer group using the kernels of the maps $\uCol_i$, we define two new Selmer groups $\Sel_p^i(f/\QQ_\infty)$. These are both $\Lambda(\Gamma)$-cotorsion, which is not true of the usual Selmer group in the supersingular case.

  Fixing a character $\eta$ of $\Delta$ and restricting to the $\eta$-isotypical component, we get maps 
  \[ \uCol^\eta_i:H^1_{\Iw}(\Qp,V)^\eta \rightarrow \Lambda_{\Qp}(\Gamma_1).\] 
  Via the Poitou-Tate exact sequence, we can reformulate Kato's main conjecture (after tensoring with $\Qp$) as follows:

  \begin{conjecture}\label{1stconj}
   For $i=1,2$, and each character $\eta$ of $\Delta$,
   \[
    \Char_{\Lambda_{\Qp}(\Gamma_1)}\Big(\Qp \otimes_{\Zp} \Sel_p^i(f/\QQ_\infty)^{\eta,\vee}\Big)=\Char_{\Lambda_{\Qp}(\Gamma_1)}\Big(\image(\uCol_i^\eta)/(L_{p,i}^\eta)\Big)
   \]
   where $M^\vee$ denotes the Pontryagin dual of a $\Lambda_{\Qp}(\Gamma_1)$-module $M$ and $\Char_{\Lambda_{\Qp}(\Gamma_1)} M$ denotes the $\Lambda_{\Qp}(\Gamma_1)$-characteristic ideal of $M$.
  \end{conjecture}

  When $v_p(a_p)$ is sufficiently large, we make use of the basis of $\NN(V)$ constructed in \cite{bergerlizhu04} to show that the first Coleman map is surjective under some additional technical conditions. Therefore, we can rewrite Conjecture~\ref{1stconj} as follows (see \cite[Corollary 6.6]{leiloefflerzerbes10}):

  \begin{theorem}\label{2ndconj}
   Under certain technical conditions, the case of $i=1$ in Conjecture~\ref{1stconj} is equivalent to the assertion that $\Char_{\Lambda_{\Qp}(\Gamma_1)}\Big(\Qp \otimes_{\Zp} \Sel_p^1(f/\QQ_\infty)^{\eta,\vee}\Big)$ is generated by $L_{p,1}^\eta$.
  \end{theorem}

  (In fact we can show that this equivalence holds integrally, i.e.~without tensoring with $\Qp$.)


 \subsection{Main results}

  In this paper, we extend the above results in several ways. Let $V$ be a crystalline representation of $\mathcal{G}_{\Qp}$ of dimension $d$ with non-negative Hodge-Tate weights. 

  We identify $\Lambda(\Gamma_1)$ with the power series ring $\Zp[[X]]$, where $X = \gamma - 1$ for a topological generator $\gamma$ of $\Gamma_1$. Denote by $\chi:\calG_{\Qp}\rightarrow \ZZ_p^\times$ the cyclotomic character. To simplify the notation, write $u=\chi(\gamma)$. 

  Firstly, we study the structure of $\Nrig(V) := \NN(V) \otimes_{\BB^+_{\Qp}} \Brig$ as a $\Gamma$-module. If $\vp^* \Nrig(V)$ denotes the $\Brig$-span of $\vp(\Nrig(V))$, then $(\vp^* \Nrig(V))^{\psi = 0}$ is contained in $(\Brig)^{\psi = 0} \otimes_{\Qp}\Dcris(V)$, and both are free $\calH(\Gamma)$-modules of rank equal to $d = \dim_{\Qp} V$. We determine the elementary divisors of the quotient of these modules:
 
  \begin{lettertheorem}[Theorem \ref{theorem:HGelem}]
  \label{lettertheorem:HGelem}
   The $\calH(\Gamma)$-elementary divisors of the quotient 
   \[
    \DD_{\cris}(V)\otimes_{\QQ_p}\calH(\Gamma) / (\vp^* \Nrig(V))^{\psi = 0}
   \]
   are $\mathfrak{n}_{r_1}, \dots, \mathfrak{n}_{r_d}$, where $r_1, \dots, r_d$ are the Hodge-Tate weights of $V$ and 
   \[
    \mathfrak{n}_k = \frac{\log(1+X)}{X}\cdots\frac{\log(u^{1-k}(1+X))}{X-u^{k-1}+1}.
   \]
  \end{lettertheorem}

  This can be seen as a $\calH(\Gamma)$-module analogue of \cite[Proposition III.2.1]{berger04}, which states that the $\Brig$-elementary divisors of the quotient 
  \[ (\Brig \otimes_{\Qp} \Dcris(V)) / \Nrig(V)\]
  are $\left(\frac{t}{\pi}\right)^{r_1},\dots,\left(\frac{t}{\pi}\right)^{r_d}$. It is striking to note that for any $k\geq 0$, the Mellin transform of $\mathfrak{n}_k$ agrees with $(1+\pi)\vp\left(\frac{t}{\pi}\right)^k$ up to a unit in $\Brig$ (see Proposition~\ref{prop:annihilator}). 

  The second aim of this paper is to use Theorem~\ref{lettertheorem:HGelem} to determine the image of the map
  \[ 1 - \vp: \Nrig(V)^{\psi = 1} \to (\vp^* \Nrig(V))^{\psi = 0}.\]
  To do this, we make use of the following commutative diagram of $\mathcal{H}(\Gamma)$-modules:
  \begin{diagram}
   \NN(V)^{\psi = 1} & \rTo^\cong_{h^1_{\Iw, V}} & H^1_{\Iw}(V)\\
   \dTo^{1 - \vp} & &  \\
   (\vp^*\Nrig(V))^{\psi = 0} & & \dTo^{\LL_V}\\
   \dInto & &  \\
   \Dcris(V) \otimes_{\Qp} (\Brig)^{\psi = 0}& \rTo^{1 \otimes \mathfrak{M}^{-1}} & \Dcris(V) \otimes_{\Qp} \calH(\Gamma)\\
  \end{diagram}

  Here the map $\LL_V$ is Perrin-Riou's regulator map, and $\mathfrak{M}: \calH(\Gamma) \rTo^\cong (\Brig)^{\psi = 0}$ denotes the Mellin transform. The commutativity of the diagram is a theorem of Berger \cite[Theorem II.13]{berger03}. Colmez's proof of the ``$\delta_V$-conjecture'' (see \cite[Theorem IX.4.4]{colmez98}), which is part of Perrin-Riou's explicit reciprocity law, gives a formula for the determinant of the matrix of $\LL_V$ (up to units). We can compare this with the determinant of the bottom left-hand map, which follows from Theorem~\ref{lettertheorem:HGelem}, to deduce that $1 - \vp: \Nrig(V)^{\psi = 1} \to (\vp^*\Nrig(V))^{\psi = 0}$ is surjective up to a small error term:

  \begin{lettertheorem}[Corollary \ref{exactsequence}]
  \label{lettertheorem:exactsequence}
   Suppose that no eigenvalue of $\vp$ on $\Dcris(V)$ lies in $p^\ZZ$. Then for each character $\eta$ of $\Delta$, there is a short exact sequence of $\calH(\Gamma_1)$-modules
   \[ 0 \rTo \NN(V)^{\psi = 1, \eta} \rTo^{1-\vp} (\vp^* \NN(V))^{\psi = 0,  \eta} \rTo^{A_\eta} \bigoplus_{i = 0}^{r_d - 1} (\Dcris(V) / V_{i, \eta})(\chi^i \chi_0^{-i} \eta) \rTo 0.\]
   Here $V_{i, \eta}$ is a subspace of $\Dcris(V)$ of the same dimension as $\Fil^{-i} \Dcris(V)$, and the map $A_{\eta}$ is the composition of the inclusion of $(\vp^*\Nrig(V))^{\psi = 0}$ in $\Dcris(V) \otimes_{\Qp} \calH(\Gamma)$ with the map $\bigoplus_i (\operatorname{id} \otimes A_{\eta, i})$, where $A_{\eta, i}$ is the natural reduction map $\calH(\Gamma) \to \Qp(\chi^i \chi_0^{-i} \eta)$ obtained by quotienting out by the ideal $(X + 1 - u^i) \cdot e_{\eta}$.
  \end{lettertheorem}

  Using this we can describe the images of the Coleman maps (for any choice of basis of $\NN(V)$):

  \begin{lettertheorem}[Corollary \ref{corollary:projectioncoleman}]
  \label{lettertheorem:images}
   Let $\eta$ be any character of $\Delta$. Then for all $1\leq i\leq d$,
   \[\image(\uCol_i^\eta)=\prod_{j\in I_i^\eta}(X-\chi(\gamma)^j+1)\Lambda_{\Qp}(\Gamma_1)\]
   for some $I_i^\eta\subset\{0,\ldots,r_d-1\}$.
  \end{lettertheorem}

  As a corollary of the proof, we also obtain a formula for the elementary divisors of the matrix of the map $\LL_V$, which can be seen as a refinement of the statement of the $\delta(V)$-conjecture. For $i\in\ZZ$, define $\ell_i =\frac{\log(1 + X)}{\log_p(\chi(\gamma))}-i$.

  \begin{lettertheorem}[Theorem \ref{theorem:LVelem}]
  \label{lettertheorem:ellV-elem}
   The elementary divisors of the $\calH(\Gamma)$-module quotient
   \[ \frac{\calH(\Gamma) \otimes_{\Qp} \Dcris(V)}{\calH(\Gamma) \otimes_{\Lambda_{\Qp}(\Gamma)}\image(\mathcal{L}_V)} \]
   are $[\lambda_{r_1}; \lambda_{r_2}; \dots; \lambda_{r_d}]$, where $\lambda_k = \ell_0 \ell_1 \dots \ell_{k-1}$.
  \end{lettertheorem}

  Suppose now that $V=V_f(k-1)$, where $f=\sum a_n q^n$ is a modular form of weight $k$ and $V_f$ is the $2$-dimensional $p$-adic representation associated to $f$. In this case, we can refine the above results to study the integral structure of the Coleman maps. Let $T_f$ be a $\calG_{\Qp}$-stable lattice in $V_f$, and let us assume that the $\BB^+_{\Qp}$-basis of $\NN(V_f)$ used to define the Coleman maps is in fact an $\AA^+_{\Qp}$-basis of $\NN(T_f)$.

  \begin{lettertheorem}[Theorem \ref{theorem:integral-images}]
  \label{lettertheorem:integral-images}
   For $i=1,2$ and for each character $\eta$ of $\Delta$, the image of $H^1_{\Iw}(\Qp, T_f)^\eta$ under $\uCol^\eta_i$ is a submodule of finite index of the module 
   \[
    \left(\prod_{j\in I_i^\eta}(X-u^j+1)\right)\Lambda(\Gamma_1)
   \]
   for some subset $I_i^\eta\subset\{0,\ldots,k-2\}$. Moreover, for each $\eta$ the sets $I_1^\eta$ and $I_2^\eta$ are disjoint.
  \end{lettertheorem}

  This theorem generalises \cite[Proposition 1.2]{kuriharapollack07}, which determines the images of $\big(\uCol^\Delta_1,\uCol^\Delta_2\big)$ for elliptic curves with $a_p = 0$. As a consequence of Theorem~\ref{lettertheorem:integral-images}, we can rewrite Conjecture~\ref{1stconj} as below, without making any technical assumptions.

  \begin{lettertheorem}
  \label{lettertheorem:1stconj2}
   For $i=1,2$, Conjecture~\ref{1stconj} is equivalent to the assertion that for each $\eta$ the characteristic ideal $\Char_{\Lambda_{\Qp}(\Gamma_1)}\Big(\Qp\otimes_{\Zp}\Sel_p^i(f/\QQ_\infty)^{\eta,\vee}\Big)$ is generated by $L_{p,i}^\eta/\prod_{j\in I_i^\eta}(X-u^j+1)$ where $I_i^\eta$ is as given by Theorem~\ref{lettertheorem:integral-images}.
  \end{lettertheorem}

  Finally, we explain in Section \ref{surjectivebasechange} how it is possible to choose a basis in such a way that $I_1^\eta=I_2^\eta=\varnothing$, i.e.~the modules $\Lambda(\Gamma_1)/\image(\uCol^\eta_i)$ are pseudo-null for both $i=1$ and $2$.

  \begin{remark}
   The local results in this paper (Theorems \ref{lettertheorem:HGelem}, \ref{lettertheorem:exactsequence}, \ref{lettertheorem:images} and \ref{lettertheorem:ellV-elem}) hold with representations of $\mathcal{G}_{\Qp}$ replaced by representations of $\mathcal{G}_{F}$ for an arbitrary finite unramified extension $F / \Qp$, with essentially the same proofs. We have chosen to work over $\Qp$ for the sake of simplicity, since this is all that is needed for applications to modular forms.
  \end{remark}


 \subsection{Setup and notation}

  \subsubsection{Fontaine rings}

   We review the definitions of the Fontaine rings we use in this paper. Details can be found in \cite{berger04} or \cite{leiloefflerzerbes10}.

   Throughout this paper, $p$ is an odd prime. If $K$ is a number field or a local field of characteristic 0, then $G_K$ denotes its absolute Galois group and $\calO_K$ the ring of integers of $K$. We write $\Gamma$ for the Galois group $\Gal(\QQ(\mu_{p^\infty}) / \QQ)$, which we identify with $\Zp^\times$ via the cyclotomic character $\chi$. Then $\Gamma \cong \Delta \times \Gamma_1$, where $\Delta$ is of order $p-1$ and $\Gamma_1\cong\Zp$. We fix a topological generator $\gamma$ of $\Gamma_1$.

   We write $\Brig$ for the ring of power series $f(\pi)\in\QQ_p[[\pi]]$ such that $f(X)$ converges everywhere on the open unit $p$-adic disc. Equip $\Brig$ with actions of $\Gamma$ and a Frobenius operator $\vp$ by $g.\pi=(\pi+1)^{\chi(g)}-1$ and $\vp(\pi)=(\pi+1)^p-1$. We can then define a left inverse $\psi$ of $\vp$ satisfying
   \[
    \vp\circ\psi(f(\pi))=\frac{1}{p}\sum_{\zeta^p=1}f(\zeta(1+\pi)-1).
   \]

   Inside $\Brig$, we have subrings $\AA_{\Qp}^+=\Zp[[\pi]]$ and $\BB_{\Qp}^+=\Qp\otimes_{\Zp}\AA_{\Qp}^+$. Moreover, the actions of $\vp$, $\psi$ and $\Gamma$ restrict to these rings. Finally, we write $t=\log(1+\pi)\in\Brig$ and $q=\vp(\pi)/\pi\in\AA_{\Qp}^+$. A formal power series calculation shows that $g(t) = \chi(g) t$ for $g \in \Gamma$ and $\vp(t) = pt$.

  \subsubsection{Iwasawa algebras and power series}

   Given a finite extension $K$ of $\Qp$, denote by $\Lambda_{\calO_K}(\Gamma)$ (respectively $\Lambda_{\calO_K}(\Gamma_1)$) the Iwasawa algebra $\ZZ_p[[\Gamma]]\otimes_{\ZZ_p}\calO_K$ (respectively $\ZZ_p[[\Gamma_1]]\otimes_{\ZZ_p}\calO_K$) over $\calO_K$. We further write $\Lambda_K(\Gamma)=\QQ\otimes\Lambda_{\calO_K}(\Gamma)$ and $\Lambda_{K}(\Gamma_1)=\QQ\otimes\Lambda_{\calO_K}(\Gamma_1)$. If $M$ is a finitely generated $\Lambda_{\calO_K}(\Gamma_1)$-torsion module, we write $\Char_{\Lambda_{\calO_K}(\Gamma_1)}(M)$ for its characteristic ideal. 

   Let
   \[ \mathcal{H}=\{f\in\QQ_p[\Delta][[X]]:\text{$f$ converges everywhere on the open unit $p$-adic disc}\},\]
   and define $\calH(\Gamma)$ to be the set of $f(\gamma-1)$ with $f(X)\in\calH$. We may identify $\Lambda_{\Qp}(\Gamma)$ with the subring of $\calH(\Gamma)$ consisting of power series with bounded coefficients. Note that $\calH(\Gamma)$ may be identified with the continuous dual of the space of locally analytic functions on $\Gamma$, with multiplication corresponding to convolution, implying that its definition is independent of the choice of generator $\gamma$.

   The action of $\Gamma$ on $\Brig$ gives an isomorphism of $\calH(\Gamma)$ with $(\Brig)^{\psi=0}$, the Mellin transform
   \begin{align*} 
    \mathfrak{M}: \calH(\Gamma) & \rightarrow (\Brig)^{\psi=0} \\
    f(\gamma-1) & \mapsto f(\gamma-1)(\pi+1). 
   \end{align*} 
   In particular, $\Lambda_{\Zp}(\Gamma)$ corresponds to $(\AA_{\Qp}^+)^{\psi=0}$ under $\mathfrak{M}$. Similarly, we define $\calH(\Gamma_1)$ as the subring of $\calH(\Gamma)$ defined by power series over $\Qp$, rather than $\Qp[\Delta]$. Then, $\calH(\Gamma_1)$ (respectively $\Lambda_{\Zp}(\Gamma_1)$) corresponds to $(1+\pi)\vp(\Brig)$ (respectively $(1+\pi)\vp(\AA_{\Qp}^+$)) under $\mathfrak{M}$.

   If $d$ is an integer and $S$ is a $\Lambda_K(\Gamma_1)$-submodule of $K\otimes_{\Qp}\calH(\Gamma_1)^{\oplus d}$ of rank $d$, we write $\det(S)$ for the determinant of any basis of $S$, which is well-defined up to multiplication by a a unit of $\Lambda_K(\Gamma_1)$. If $F$ is some $\Lambda_K(\Gamma_1)$-homomorphism with image, which is free of rank $d$ over $\Lambda_L(\Gamma_1)$, lying inside $K\otimes_{\Qp}\calH(\Gamma_1)^{\oplus d}$, we write $\det(F)$ for $\det(\image(F))$.

   For an integer $i$, define
   \[
    \left.
    \begin{aligned}
     \ell_i &=\frac{\log(1 + X)}{\log_p(\chi(\gamma))}-i \\ 
     \delta_i &= \frac{\ell_i}{X + 1 - \chi(\gamma)^i}
    \end{aligned}
    \right\} \in \calH(\Gamma_1).\\
   \]
   Note that $\ell_i$ is independent of the choice of generator $\gamma$ (hence the choice of normalising factor), but $\delta_i$ is not.

   \begin{remark}
    Note that for any positive integer $k$, we have
    \[
     \mathfrak{n}_k = 
     a_k \delta_{k-1} \dots \delta_0,
    \]
    where $a_k=\log(\chi(\gamma))^k \in \Zp$ is nonzero.
   \end{remark}

   The following result slightly refines \cite[Lemma II.2]{berger03}.

   \begin{proposition}\label{prop:annihilator}
    For any $k \ge 0$, we have
    \begin{align*}
     \mathfrak{M}( \ell_{k-1} \dots \ell_0 \calH(\Gamma) ) &= (t^k \Brig)^{\psi = 0}\\
     \mathfrak{M}( \delta_{k-1} \dots \delta_0 \calH(\Gamma) ) &= \left( \left(\tfrac{t}{\vp(\pi)}\right)^k \Brig\right)^{\psi = 0}.
    \end{align*} 
   \end{proposition}
  
   \begin{proof}
    One checks easily that $\ell_i$ acts on $\Brig$ as the differential operator $(1 + \pi) t \frac{\mathrm{d}}{\mathrm{d}\pi} - i$ and hence
    \[ \ell_j( t^j f ) = t^{j+1}(1 + \pi)\frac{\mathrm{d}f}{\mathrm{d}\pi}.\]
    Since $(1 + \pi)\frac{\mathrm{d}}{\mathrm{d}\pi}$ is an isomorphism on $(\Brig)^{\psi = 0}$ (it is the map on distributions dual to the map $f(x) \mapsto xf(x)$ on functions), it follows that each $\ell_j$ maps $(t^j \Brig)^{\psi = 0}$ bijectively onto $(t^{j+1} \Brig)^{\psi = 0}$. 

    To prove the corresponding statement for the $\delta_i$, we note that $(\Brig / \vp(\pi) \Brig)^{\psi = 0}$ is isomorphic to $\Qp[\Delta]$ as a $\Gamma$-module. Since $t$ is a uniformiser of the ideal $\vp(\pi)$, we have $(\vp(\pi)^j\Brig / \vp(\pi)^{j+1} \Brig)^{\psi = 0} = (t^j \Brig / t^j \vp(\pi) \Brig)^{\psi = 0} \cong \Qp[\Delta](j)$ as a $\Gamma$-module. Hence the annihilator of the $\calH(\Gamma)$-module
    \[
     (\Brig / \vp(\pi)^{k} \Brig)^{\psi = 0} \cong (t^k \Brig)^{\psi = 0} / \left( \left(\tfrac{t}{\vp(\pi)}\right)^k \Brig\right)^{\psi = 0}
    \]
    is $X (X + 1 - \chi(\gamma)) \dots (X + 1 - \chi(\gamma)^{k-1})$. This is coprime to $\delta_{0} \dots \delta_{k-1}$ and the product is $\ell_0 \dots \ell_{k-1}$, so the result follows.
   
   \end{proof}


  \subsubsection{Isotypical components}

   Let $\eta:\Delta\rightarrow\Zp^\times$ be a character. We write $e_\eta=(p-1)^{-1}\sum_{\sigma\in\Delta}\eta^{-1}(\sigma)\sigma$. If $M$ is a $\Lambda_E(\Gamma)$-module, its $\eta$-isotypical component is given by $M^\eta=e_\eta\Lambda_E(\Gamma)$. When $\eta=1$, we write $M^\Delta$ in place of $M^\eta$.

   We identify $\Lambda_E(\Gamma_1)$ with the power series in $X=\gamma-1$ with bounded coefficients in $E$. Given
   \[
    F=\sum_{\sigma\in\Delta,n\ge0}a_{\sigma,n}\sigma(\gamma-1)^n\in\Lambda(\Gamma),\]
   we write $F^\eta=e_\eta F$ for its image in $\Lambda_E(\Gamma)^\eta$. In particular,
   \[
    F^\eta=e_{\eta}\sum_{n\ge0}\left(\sum_{\sigma\in\Delta}a_{\sigma,n}\eta(\sigma)\right)(\gamma-1)^n\in e_\eta\Lambda_E(\Gamma_1).
   \]
   Therefore, we can identify $F^\eta$ with a power series in $X=\gamma-1$. Under this identification, the value $F^\eta|_{X=\chi(\gamma)^j-1}$ is given by $\chi^j\chi_0^{-j}\eta(F)$ where $\chi_0=\chi|_{\Delta}$ for all $j\in\ZZ$.


  \subsubsection{Crystalline representations}\label{sect:crystallinereps}

   Let $E$ and $F$ be finite extensions of $\Qp$. Let $V$ be a crystalline $E$-linear representation of $G_{\Qp}$.  We denote the Dieudonn\'{e} module of $V$ by $\Dcris(V)$. If $j\in\ZZ$, $\Fil^j\Dcris(V)$ denotes the $j$th step in the de Rham filtration of $\Dcris(V)$. We say $V$ is \emph{positive} if $\Dcris(V) = \Fil^0 \Dcris(V)$ (following the standard, but unfortunate, convention that positive representations are precisely those with non-positive Hodge-Tate weights).

   The $(\vp,\Gamma)$-module of $V$ is denoted by $\DD(V)$. As shown by Fontaine (unpublished -- for a reference see~\cite[Section II]{cherbonniercolmez99}), we have a canonical isomorphism of $\Lambda_E(\Gamma)$-modules
   \[
     h^1_{\Iw,V}: \DD(V)^{\psi=1} \rightarrow \HIw(\Qp,V).
   \]
   We write $\exp_{F,V}:F\otimes\Dcris(V)\rightarrow H^1(F,V)$ for Bloch-Kato's exponential over $F$.

   For an integer $j$, $V(j)$ denotes the $j$th Tate twist of $V$, i.e. $V(j)=V\otimes E e_j$ where $G_{\Qp}$ acts on $e_j$ via $\chi^j$. We have
   \[
    \Dcris(V(j))=t^{-j}\Dcris(V)\otimes e_j.
   \]
   For any $v\in\Dcris(V)$, $v_j=t^{-j}v \otimes e_j$ denotes its image in $\Dcris(V(j))$. 

   If $h\ge1$ is an integer such that $\Fil^{-h}\Dcris(V)=\Dcris(V)$, we write $\Omega_{V,h}$ for the Perrin-Riou exponential as defined in \cite{perrinriou94}.

   Let $T$ be an $\calO_E$-lattice which is stable under $G_{\Qp}$. We denote the Wach module of $V$ (respectively $T$)  by $\NN(V)$ (respectively $\NN(T)$), a free module of rank $d$ over $\BB_{\Qp}^+$ (respectively $\AA_{\Qp}^+$). Recall that $\Gamma$ acts on both of these objects, and there is a map $\vp: \NN(T)[\pi^{-1}] \to \NN(T)[\vp(\pi)^{-1}]$, preserving $\NN(T)$ if $T$ is positive (and similarly for $V$).

   For any $j \in \ZZ$ we can identify $\NN(T(j))$ with $\pi^{-j} \NN(T) \otimes e_j$, where $e_j$ is a basis of $\Qp(j)$. Given a $R$-module $M$ with an action of $\vp$ and a submodule $N$, $\vp^* N$ denotes the $R$-submodule of $M$ generated by $\vp(N)$, e.g.~$\vp^*\NN(T)$ denotes the $\AA_{\Qp}^+$-submodule of $\NN(T)[\pi^{-1}]$ generated by $\vp(\NN(T))$. Finally, we write $\Nrig(V)=\NN(V)\otimes_{\BB_{\Qp}^+}\Brig$. 

   The following lemma is implicit in the calculations of \cite[\S 3]{leiloefflerzerbes10}, but for the convenience of the reader we give a separate proof:

   \begin{lemma}\label{lemma:vpstar}
    If the Hodge-Tate weights of $V$ are $\ge 0$, then we have
    \[ \NN(T) \subseteq \vp^* \NN(T)\]
    and similarly for $\NN(V)$.
   \end{lemma}

   \begin{proof}
    It suffices to prove the result for $T$. Suppose that the Hodge-Tate weights of $V$ are in $[0, m]$. Then $\NN(T) = \pi^{-m}\NN(T(-m))$. Since $T(-m)$ is positive, $\vp$ preserves $\NN(T(-m))$ and $\NN(T(-m)) / \vp^*\NN(T(-m))$ is killed by $q^m$ \cite[proof of Theorem III.3.1]{berger04}. Equivalently, we have
    \[ q^m \cdot \pi^m \NN(T) \subseteq \vp^*(\pi^m \NN(T)) = \vp(\pi)^m \vp^*\NN(T).\]
    Since $q = \vp(\pi) / \pi$, the result follows.
   \end{proof}


  \subsubsection{Modular forms}\label{sect:modularforms}

   Let $f=\sum a_nq^n$ be a normalised new eigenform of weight $k\ge2$, level $N$ and nebentypus $\varepsilon$. Write $F_f=\QQ(a_n:n\ge1)$ for its coefficient field. Let $\bar{f}=\sum\bar{a}_nq^n$ be the dual form to $f$, then $F_f=F_{\bar{f}}$. We assume that $p\nmid N$ and $a_p$ is not a $p$-adic unit, so $f$ is supersingular at $p$. 
   
   \begin{remark}
    We make this assumption in save ourselves from doing the same calculations twice in Section~\ref{imageMF}; they easily generalise to the ordinary case.
   \end{remark}

   We fix a prime of $F_f$ above $p$. We denote the completion of $F_f$ at this prime by $E$ and fix a uniformiser $\varpi_E$. We write $V_f$ for the 2-dimensional $E$-linear representation of $G_{\QQ}$ associated to $f$ from \cite{deligne69}. We fix an $\calO_E$-lattice $T_f$ stable under $G_{\QQ}$, it then determines a lattice $T_{\bar{f}}$ of $V_{\bar{f}}$. When restricted to $G_{\Qp}$, $V_f$ is crystalline and its de Rham filtration is given by
   \[
    \Fil^i\Dcris(V_f)=
    \left\{
    \begin{array}{ll}
     \Dcris(V_f)         & \text{if $i\le0$}\\
     E\omega                     & \text{if $1\le i\le k-1$}\\
     0                          & \text{if $i\ge k$}
    \end{array}\right.
   \]
   for some $0\ne\omega\in\Dcris(V_f)$. The action of $\vp$ on $\Dcris(V_f)$ satisfies $\vp^2-a_p\vp+\varepsilon(p)p^{k-1}=0$. We write $\nu_1,\nu_2$ for the `good basis' of $\Dcris(V_f)$ as chosen in \cite[\S 3.3]{leiloefflerzerbes10}.

   Unless otherwise stated, we always assume that the eigenvalues of $\vp$ on $\Dcris(V_f)$ are not integral powers of $p$ and the nebentypus of $f$ is trivial. Our assumption on the eigenvalues of $\vp$ allows us to define the Perrin-Riou pairing
   \[
    \LL_{i}=\LL_{1,(1+\pi)\otimes\nu_{i,1}}:H^1_{\Iw}(\Qp,V_{\bar{f}}(k-1))\rightarrow\calH(\Gamma)
   \]
   for $i=1,2$ where we have identified $V_f(1)^*(1)$ with $V_{\bar{f}}(k-1)$ (see \cite[Section~3.2]{lei09} or \cite[Section~3.3]{leiloefflerzerbes10} for details).


  \subsubsection{Elementary divisors}

   Let $R$ be a commutative ring with $1$. We say that $R$ admits the theory of elementary divisors if the following holds. Let $M\le N$ be finitely generated $R$-modules such that $N\cong R^d$, then there exists a $R$-basis $n_1,\ldots,n_d$ of and $r_1,\ldots,r_d\in R$ (unique up to a unit of $R$) such that $r_1 \mid \cdots \mid r_d$ and $r_1n_1,\ldots,r_en_e$, where $e$ is the largest integer such that $r_e\ne0$, form a $R$-basis of $M$. In this case, we write $[N:M]=[N:M]_R=[r_1;\cdots;r_d]$. When $d=1$, we simply write $[N:M]=r_1$.

   In particular, as explained in \cite{berger04}, $\Brig$ admits the theory of elementary divisors, as it is a Bezout ring (all finitely generated ideals are principal). The same can be said of $E\otimes_{\Qp}\Brig$ for any finite extension $E$ of $\Qp$, and of $\mathcal{H}(\Gamma_1)$ (which is isomorphic to $\Brig$ as an abstract ring).

   We will need the following lemma; see \cite[Lemma~III.7.6]{lang02}.

   \begin{lemma}
    Let $R$ be an elementary divisor ring, and suppose $M$ is a finitely-presented $R$-module and $N$ a submodule such that $N \cong R / a$ for some $a \in R$ and $aM = 0$. Then $M \cong N \oplus M/N$.
   \end{lemma}

   \begin{proof}
    Let $q_1, \dots, q_r$ be a set of generators for $M/N$, with annihilators $a_i$, giving an isomorphism $M/N \cong \oplus_{i=1}^r R/a_i$. It is easy to check that each $a_i$ divides $a$. Let $p_i$ be an arbitrary lift of $q_i$; then $a_i p_i \in N$, so $a_i p_i = b_i p_0$ where $p_0$ is a generator of $N$ and $b_i \in R / aR$. Since $aM = 0$, we have $0 = (a / a_i) a_i p_i = (a/a_i) b_i p_0$. Thus $a_i \mid b_i$, and we may write $b_i = a_i c_i$. Thus $p_i' = p_i - c_i p_0$ is a lift of $p_i$ such that $a_i p_i' = a_i p_i - a_i c_i p_0 = a_i p_i = b_i p_0 = 0$. It follows that the subgroup generated by the $p_i'$ maps bijectively to $M/N$, giving the required splitting.
   \end{proof}

   A straightforward induction gives the following generalisation: 

   \begin{corollary}\label{corollary:elem-divisors}
    If $M$ is an $R$-module with a filtration by submodules $0 = M_0 \subseteq M_1 \subseteq \dots \subseteq M_d = M$, and there are elements $a_1, \dots, a_d \in R$ such that for each $i = 1, \dots,d$ we have $M_i / M_{i-1} \cong R/a_i$ and $a_i M \subseteq M_{i-1}$, then $M \cong \bigoplus_{i=1}^d  R/a_i$.
   \end{corollary}

   The ring $\calH(\Gamma)$ is not a domain; but it is equal to the direct sum of its subrings $e_\eta \calH(\Gamma)$, where $e_\eta$ is the idempotent in $\Qp[\Delta]$ corresponding to the character $\eta: \Delta \to \Qp^\times$ as above. Each of these subrings is isomorphic to $\calH(\Gamma_1)$, and hence admits a theory of elementary divisors.


 \section{Refinements of crystalline representations and \texorpdfstring{$\calH(\Gamma)$}{H(Gamma)}-structure}
 
  In this section, we will prove Theorem \ref{lettertheorem:HGelem}. We will do this by working with a certain filtration of the module $\Nrig(V)$, which is a $(\vp, \Gamma)$-module over $\Brig$; the steps in this filtration are $(\vp, \Gamma)$-modules over $\Brig$, but they are not necessarily of the form $\Nrig(W)$ for any representation $W$, so we begin by systematically developing a theory of such modules. Our approach is very much influenced by the description of the theory of $(\vp, \Gamma)$-modules over the Robba ring $\Bdag$ given in \cite[\S 2.2]{bellaichechenevier09}.


  \subsection{Some properties of \texorpdfstring{$(\vp, \Gamma)$}{(phi,Gamma)}-modules over \texorpdfstring{$\Brig$}{B+rig}}

   Let $\calN$ be a module over $\Brig$, free of rank $d$ and endowed with semilinear commuting actions of $\vp$ and $\Gamma$. We define
   \[ \Dcris(\calN) = \calN^\Gamma.\]
   We equip $\Dcris(\calN)$ with the filtration defined by 
   \[ \Fil^i \Dcris(\calN) = \{ v \in \Dcris(\calN) : \vp(v) \in q^j \calN\}\] 
   where $q = \vp(\pi)/\pi$ as above.

   Let $K_n = \Qp(\mu_{p^n})$ and $K_\infty = \bigcup_{n} K_n$. We define
   \[ \DdR^{(n)}(\calN) = \left(K_\infty \otimes_{K_n} K_n[[t]] \otimes_{\Brig} \calN\right)^\Gamma,\]
   where the tensor product is via the embedding $\Brig \into K_n[[t]]$ arising from the fact that 
\[K_n \cong \Brig / \vp^{n-1}(q)\]
 and $t$ is a uniformiser of the prime ideal $\vp^{n-1}(q)$. We endow $K_n[[t]]$ with the obvious semilinear action of $\Gamma$, for which this homomorphism is $\Gamma$-equivariant, and the $t$-adic filtration. Then $\DdR^{(n)}(\calN)$ is a filtered $\Qp$-vector space, of dimension $\le d$ (since $K_\infty((t))^\Gamma = \Qp$ \cite[\S 2.2.7]{bellaichechenevier09}); the operator $\vp$ gives an isomorphism of filtered $\Qp$-vector spaces $\DdR^{(n)}(\calN) \rTo^\cong \DdR^{(n+1)}(\calN)$ for each $n$, and an embedding of filtered $\Qp$-vector spaces $\Dcris(\calN) \into \DdR^{(1)}(\calN)$.

   We shall say that $\calN$ is \emph{crystalline} if $\dim_{\Qp} \Dcris(\calN) = d$, and \emph{de Rham} if $\dim_{\Qp} \DdR^{(n)}(\calN) = d$ (for some, and hence all, $n \ge 1$). If $\calN$ is de Rham, we define the \emph{Hodge-Tate weights} of $\calN$ to be the integers $r$ such that $\Fil^{-r} \DdR^{(n)}(\calN) \ne \Fil^{1-r} \DdR^{(n)}(\calN)$ (with multiplicities given by the size of the jump in dimension). Note that these are necessarily $\le 0$, which is unfortunate but necessary for compatibility with the usual definition in the case of Galois representations.

   Finally, we define $\DSen^{(n)}(\calN) = K_\infty \otimes_{K_n} \calN/\vp^{n-1}(q)\calN$. This is a $K_\infty$-vector space of dimension $d$, with a semilinear action of $\Gamma$. As above, the $\vp$ operator gives isomorphisms $\DSen^{(n)}(\calN) \to \DSen^{(n+1)}(\calN)$, of $K_\infty$-vector spaces with semilinear $\Gamma$-action. (So both $\DSen(\calN)$ and $\DdR(\calN)$ are independent of $n$ as abstract objects; we retain the $n$ in the notation when we are interested in the relation between these spaces and the original module $\calN$.)

   \begin{proposition}
    Let $j \ge 0$, and suppose $\calN$ is de Rham. Then there is an isomorphism of $\Qp$-vector spaces
    \[\Fil^j \DdR(\calN) / \Fil^{j+1} \DdR(\calN) \rTo^\cong \DSen(\calN)^{\Gamma = \chi^{-j}}.\]
   \end{proposition}

   \begin{proof}
    Let us fix an $n \ge 1$ and let $\theta$ be the reduction map $K_n[[t]] \to K_n$. Then $\theta$ induces a map
    \[ \DdR^{(n)}(\calN) \to \DSen^{(n)}(\calN)^\Gamma \]
    whose kernel is $\Fil^1 \DdR(\calN)$ and whose image is a $\Qp$-linear subspace $S_0 \subseteq \DSen(\calN)^\Gamma$. Similarly, we find that $\theta \circ t^{-j}$ gives an injection $\Fil^j \DdR(\calN) / \Fil^{j+1} \DdR(\calN) \to \DSen(\calN)^{\Gamma = \chi^{-j}}$, whose image is a $\Qp$-linear subspace $S_j$.

    Since $\bigoplus_{j=0}^\infty S_j$ has dimension $d$, it suffices to show that $\dim_{\Qp} \bigoplus_{j = 0}^\infty \DSen(\calN)^{\Gamma = \chi^{-j}} \le d$. This follows from the fact that it is a subspace of $\left(K_\infty((t)) \otimes_{K_\infty} \DSen(\calN)\right)^\Gamma$, and (as remarked above) $K_\infty((t)))$ is a field, with $K_\infty((t))^\Gamma = \Qp$ .
   \end{proof}

   \begin{corollary}\label{corollary:invariants}
    If $\calN$ is crystalline, then the map
    \[ \Dcris(\calN) = \calN^\Gamma \rTo^{\vp^{n}} (\calN / \vp^{n-1}(q)^r \calN)^\Gamma \]
    is surjective for all $r \ge 1$ and $n \ge 1$, with kernel $\Fil^r \Dcris(\calN)$.
   \end{corollary}

   \begin{proof}
    Let us define $\calN^{(n)} = K_\infty \otimes_{K_n} K_n[[t]] \otimes_{\Brig} \calN$, so $(\calN^{(n)})^\Gamma = \DdR(\calN)$. By hypothesis the map $\vp^n: \Dcris(\calN) \to \DdR^{(n)}(\calN)$ is an isomorphism of filtered vector spaces, and the filtration on $\DdR(\calN)$ is defined by the $t$-adic filtration of $\calN^{(n)}$, so it suffices to show that reduction modulo $t^r$ gives a surjection
    \[ (\calN^{(n)})^\Gamma \to (\calN^{(n)} / t^r \calN^{(n)})^\Gamma.\]

    We show that for each $j$, the map $(t^j \calN^{(n)})^\Gamma \to (t^j \calN^{(n)} / t^{j+1}\calN^{(n)})^\Gamma$ is surjective. 
    Multiplication by $t^{-j}$ gives an isomorphism $(t^j \calN^{(n)} / t^{j+1}\calN^{(n)})^\Gamma \to (\calN^{(n)} / t \calN^{(n)})^{\Gamma = \chi^{-j}}$; 
    but $\calN^{(n)} / t \calN^{(n)} = \DSen^{(n)}(\calN)$, and by the preceding proposition we know that $\theta \circ t^{-j}$ gives an 
    isomorphism from $\Fil^j \DdR(\calN) / \Fil^{j+1} \DdR(\calN)$ to $\DSen^{(n)}(\calN)^{\Gamma = \chi^{-j}}$. So the map $(\calN^{(n)})^\Gamma \to (\calN^{(n)} / t^r \calN^{(n)})^\Gamma$ is an injection of filtered modules for which the associated map of graded modules is surjective; thus it is itself surjective. 
   \end{proof}

   Let us write $\mathcal{M} = \Brig \otimes_{\Qp} \Dcris(\calN) \subseteq \calN$.

   \begin{proposition}\label{prop:elementary-divisors}
    If $\calN$ is crystalline and $\Gamma$ acts trivially on $\calN / \pi \calN$, then the elementary divisors of $\calN / \mathcal{M}$ are 
    \[ \left(\tfrac{t}{\pi}\right)^{s_1}, \dots, \left(\tfrac{t}{\pi}\right)^{s_d},\]
    where $-s_1 \ge \dots \ge -s_d$ are the Hodge-Tate weights of $\calN$.
   \end{proposition}

   \begin{proof} This follows exactly as in \cite[Proposition III.2.1]{berger04}. \end{proof}


  \subsection{Quotients of \texorpdfstring{$(\vp, \Gamma)$}{(phi,Gamma)}-modules}

   We now let $\calN$ be a $(\vp, \Gamma)$-module over $\Brig$, as above. We assume that $\calN$ is crystalline and $\Gamma$ acts trivially on $\calN / \pi \calN$, and investigate the properties of a certain class of $(\vp, \Gamma)$-modules obtained as quotients of $\calN$. We continue to write $\mathcal{M} = \Brig \otimes_{\Qp} \Dcris(\calN) \subseteq \calN$.

   Let $Y$ be a $\vp$-stable $E$-linear subspace of $\Dcris(\calN)$. We set 
   \[\mathcal{Y} = \Brig \otimes_{\Qp} Y \subseteq \mathcal{M}.\]
   and
   \[\mathcal{X} = \calN \cap \mathcal{Y}\left[\left(\tfrac{t}{\pi}\right)^{-1}\right] = \{ x \in \calN : \left(\tfrac{t}{\pi}\right)^m x \in \mathcal{Y} \text{ for some $m$}\} \subseteq \mathcal{N}.\]

   \begin{proposition}\label{prop:basic-properties} The spaces $Y, \mathcal{Y}, \mathcal{X}$ have the following properties:
    \begin{enumerate}
     \item[(a)] $\mathcal{X}$ is a $\Brig$-submodule of $\calN$ stable under $\vp$ and $\Gamma$;
     \item[(b)] $\mathcal{X} = \{ x \in \calN : a x \in \mathcal{Y}$ for some nonzero $a \in \Brig\}$ (the saturation of $\mathcal{Y}$);
     \item[(c)] $\mathcal{X}$ is free of rank $\dim_{\Qp} Y$ as an $\Brig$-module;
     \item[(d)] $Y = \mathcal{X} \cap \Dcris(\calN)$ and $\mathcal{Y} = \mathcal{X} \cap \mathcal{M}$.
    \end{enumerate}
   \end{proposition}
    
   \begin{proof}
    Part (a) is immediate from the definition. 
   
    For (b), suppose $x \in \calN$ and there is some nonzero $a \in \Brig$ such that $ax \in \mathcal{Y}$. By Proposition~\ref{prop:elementary-divisors}, we can find $m$ such that $\left(\tfrac{t}{\pi}\right)^m x \in \mathcal{M}$, and $a \left(\tfrac{t}{\pi}\right)^m x \in \mathcal{Y}$. Since $\mathcal{Y}$ is clearly saturated in $\mathcal{M}$, we deduce that $\left(\tfrac{t}{\pi}\right)^m x \in \mathcal{Y}$, and hence $x \in \mathcal{X}$ as required.
    
    For part (c), it is immediate that $\mathcal{X}\left[\left(\tfrac{t}{\pi}\right)^{-1}\right]$ is free of rank $\dim_{\Qp} Y$ as a $\Brig\left[\left(\tfrac{t}{\pi}\right)^{-1}\right]$-module. Hence the rank of $\mathcal{X}$ over $\Brig$ must be equal to this. Also, since $\mathcal{X}$ is a submodule of the free module $\calN$, it is torsion-free. Hence it is free.
    
    For part (d), it is clear that $\mathcal{Y} \subseteq \mathcal{X} \cap \mathcal{M}$; and this inclusion is an equality, since $\mathcal{M} / \mathcal{Y}$ is torsion-free and $\mathcal{X} / \mathcal{Y}$ is torsion. Since $\mathcal{Y} \cap \Dcris(\calN) = Y$, the statement follows.
   \end{proof}

   Let $W = \Dcris(\calN) / Y$ and $\mathcal{W} = \calN / \mathcal{X}$. The quotient $\mathcal{W}$ is a $(\vp, \Gamma)$-module in the obvious way, and the natural map $W \into \Dcris(\mathcal{W})$ is clearly injective; hence it is also surjective, for reasons of dimension.

   \begin{proposition}\label{prop:quotient-filtration}
    The quotient filtation $\Fil^\bullet W$ induced on $W$ by the filtation of $\Dcris(N)$ agrees with the filtation $\FFil$ given by
    \[ \FFil^r W = \{ w \in W : \vp(w) \in q^r \mathcal{W}\}.\]
   \end{proposition}

   \begin{proof}
    It is clear from the definition that $\Fil^r W \subseteq \FFil^r W$. 

    Conversely, let $y \in \Dcris(N)$ such that $[y]\in\FFil^r W$, so we can write $\vp(y) = q^r y' + z$ for some $y' \in N$ and $z \in \mathcal{X}$. Then 
    \[ z \bmod q^r \mathcal{X} \in (\mathcal{X} / q^r \mathcal{X})^{\Gamma}.\]
    Applying Corollary~\ref{corollary:invariants} to $\mathcal{X}$, we find that $z$ is congruent modulo $q^r$ to an element of $\mathcal{X}^\Gamma = Y$.
   \end{proof}

   The final result we will need about these quotients is the following slightly fiddly lemma. Let us suppose that the jumps in the filtration of $\Dcris(\calN)$, with multiplicity, are $s_1 \le s_2 \dots \le s_d$ (i.e.~the Hodge-Tate weights of $\calN$ are $-s_i$). We say that the $\vp$-stable subspace $Y$ is \emph{in general position} (with respect to the Hodge filtration of $\Dcris(\calN)$) if the jumps in the filtration $\Fil^\bullet Y$ are $s_1, \dots, s_j$, where $j = \dim_{\Qp} Y$.

   \begin{lemma}\label{lemma:fiddlylemma} If $Y$ is in general position, then for any $m \ge s_d$, we have
    \[ \left(\tfrac{t}{\pi}\right)^{m - s_{(j+1)}} \mathcal{M}  \subseteq \left(\tfrac{t}{\pi}\right)^{m} \calN + \mathcal{Y}.\]
   \end{lemma}

   \begin{proof}
    Consider the quotient module $\mathcal{W} = \calN / \mathcal{X}$. This is a crystalline $(\vp, \Gamma)$-module over $\Brig$ of rank $d - j$ with $\Gamma$ trivial modulo $\pi$, and its Hodge-Tate weights are exactly $\{s_{(j+1)}, \dots, s_d\}$; hence its $\Gamma$-invariants lie in $\left(\tfrac{t}{\pi}\right)^{s_{(j+1)}} \mathcal{W}$. This is equivalent to 
    \( \mathcal{M} \subseteq \left(\tfrac{t}{\pi}\right)^{s_{(j+1)}} \calN + \mathcal{X}\).
    Multiplying by $(\tfrac{t}{\pi})^{m - s_{(j+1)}}$, we see that
    \[ \left(\tfrac{t}{\pi}\right)^{m - s_{(j+1)}} \mathcal{M} \subseteq \left(\tfrac{t}{\pi}\right)^{m} \calN + \left(\tfrac{t}{\pi}\right)^{m - s_{(j+1)}}\mathcal{X}.\]
    Since both $\left(\tfrac{t}{\pi}\right)^{m - s_{(j+1)}} \mathcal{M}$ and $\left(\tfrac{t}{\pi}\right)^{m} \mathcal{N}$ are manifestly contained in $\mathcal{M}$, we may replace the last term with its intersection with $\mathcal{M}$, which is clearly contained in $\mathcal{X} \cap \mathcal{M} = \mathcal{Y}$.
   \end{proof}

  \subsection{Application to crystalline representations}
  
   Let $V$ be a $d$-dimensional crystalline representation of $G_{\Qp}$ with Hodge-Tate weights $\{-s_1, \dots, -s_d\}$, where $0 \le s_1 \le \dots \le s_d$ (so $V$ is positive in the sense of \S \ref{sect:crystallinereps} above). As above, we define $\Nrig(V) = \Brig \otimes_{\mathbb{B}^+_{\mathbb{Q}_p}} \mathbb{N}(V)$, where $\NN(V)$ is the Wach module of $V$ as constructed in \cite{berger04}. Then $\Nrig(V)$ is a crystalline $(\vp, \Gamma)$-module over $\Brig$ with $\Gamma$ trivial modulo $\pi$, and $\Dcris(V)$ is isomorphic (as a filtered $\vp$-module over $\Qp$) to $\Dcris(\Nrig(V))$ as defined in the previous section \cite[Theorems II.2.2 and III.4.4]{berger04}.

   If $V$ is in fact an $E$-linear representation, for $E$ some finite extension of $\Qp$, then $\Nrig(V)$ is naturally an $E \otimes_{\Qp} \Brig$-module, and $\Dcris(V)$ is a filtered $E$-vector space. If we choose an $E$-linear $\vp$-stable subspace, then all of the above constructions commute with the additional $E$-linear structure.

   As in \cite[\S 2.4.1]{bellaichechenevier09}, we define a \emph{refinement} of $V$ to be a family $\underline{Y} = (Y_i)_{i=1}^d$ of $\vp$-stable $E$-linear subspaces of $\Dcris(V)$, with $0 \subsetneq Y_1 \subsetneq \dots \subsetneq Y_d = \Dcris(V)$, so $\dim_E Y_i = i$. We say that the refinement is \emph{non-critical} \cite[definition 2.4.5]{bellaichechenevier09} if $\Dcris(V) = Y_i \oplus \Fil^{s_i + 1} \Dcris(V)$ for each $i$; as shown in \cite[Proposition 2.4.7]{bellaichechenevier09}, this is true if and only if each of the spaces $Y_i$ is in general position in the above sense, so the jumps of the filtration obtained by intersecting with $Y_i$ are $\{s_1, \dots, s_i\}$.

   We shall suppose that $V$ admits a non-critical refinement, and fix a choice of such a refinement $\underline{Y}$. As in the previous subsection, the $\vp$-stable subspaces $Y_i$ determine submodules $\mathcal{Y}_i = \Brig \otimes_{\Qp} Y_i \subseteq \mathcal{M}$ and $\mathcal{X}_i = \mathcal{Y}_i^{\mathrm{sat}}$.
   
   Let us consider the representation $V(m)$, for some $m \ge s_d$. This has non-negative Hodge-Tate weights $\{m - s_i\}_{i = 1, \dots, d}$. If $e_m$ denotes a basis of $\Qp(m)$, then we have
   \[ \begin{aligned}
       \Dcris(V(m)) &= \{t^{-m} x \otimes e_m : x \in \Dcris(V)\},\\
       \Nrig(V(m)) &= \{ \pi^{-m} y \otimes e_m : y \in \Nrig(V)\}.
      \end{aligned}
   \]

   We define $\mathcal{A}_i = \{ \pi^{-m} y \otimes e_m : y \in \mathcal{X}_i\}$ and $\mathcal{B}_i = \{ t^{-m} x \otimes e_m : x \in \mathcal{Y}_i\}$.
   \begin{proposition}
    For each $i = 1, \dots, d$,
    \begin{enumerate}
     \item[(a)] $(\tfrac{t}{\pi})^{m - s_i} \mathcal{B}_i \supseteq \mathcal{A}_i \supseteq (\tfrac{t}{\pi})^{m-s_1} \mathcal{B}_i$;
     \item[(b)] $\mathcal{B}_i$ is the saturation of $\mathcal{A}_i$ in $\mathcal{B}_d = \Brig \otimes \Dcris(V(m))$;
     \item[(c)] The inclusion $\mathcal{A}_d \into \mathcal{B}_d$ identifies $\mathcal{A}_d / \mathcal{A}_{i-1}$ with a submodule of $\mathcal{B}_d / \mathcal{B}_{i-1}$ and the quotient is annihilated by $(\tfrac{t}{\pi})^{m - s_i}$.
    \end{enumerate}
   \end{proposition}

   \begin{proof}
    The chain of inclusions in (a) is equivalent to $(\tfrac{t}{\pi})^{s_1} \mathcal{X}_i \supseteq \mathcal{Y}_i \supseteq (\tfrac{t}{\pi})^{s_i} \mathcal{X}_i$, which follows from Propostion~\ref{prop:elementary-divisors} since the Hodge-Tate weights of $\mathcal{X}_i$ are $\{-s_1, \dots, -s_i\}$. Moreover, $\mathcal{B}_i$ is manifestly saturated in $\mathcal{B}_d$ (being the base extension of a subspace of $\Dcris(V(m))$), and together with (a), this proves (b).
    For part (c), we note that $\mathcal{A}_d \cap \mathcal{B}_{i-1} = \mathcal{A}_{i-1}$, so the given map is well-defined and injective; to show that the annihilator is as claimed, we must check that 
    \[ (\tfrac{t}{\pi})^{m - s_i} \mathcal{B}_d \subseteq \mathcal{B}_{i-1} + \mathcal{A}_d,\]
    which is equivalent to Lemma~\ref{lemma:fiddlylemma}.
   \end{proof}

   We now pass from the ``additive'' to the ``multiplicative'' situation. Let us define $\widetilde{\mathcal{A}}_i = \bigoplus_{s=1}^{p-1} (1 + \pi)^s \vp(\mathcal{A}_i)$, and similarly for $\widetilde{\mathcal{B}}_i$. Note that these are $\Gamma$-stable, since $\Gamma$ and $\vp$ commute. Moreover, the action of $\Gamma$ on $\widetilde{\mathcal{B}}$ clearly extends to an action of the ring $\calH(\Gamma)$, which is continuous with respect to the Fr\'echet topologies of $\calH(\Gamma)$ and $\widetilde{\mathcal{B}}$. As the submodules $\mathcal{B}_i$ and $\mathcal{A}_i$ are all clearly closed and $\Gamma$-invariant, they also inherit a Fr\'echet topology and a continuous action of $\calH(\Gamma)$.

   \begin{remark} Note that we can define an operator $\psi: \mathcal{B}_d \to \mathcal{B}_d$ which is $\vp^{-1}$ on $\Dcris(V)$ and is $\Brig$-semilinear (for the usual definition of $\psi$ acting on $\Brig$). Then $\widetilde{\mathcal{A}}_i = (\vp^* \mathcal{A}_i)^{\psi = 0}$, where $\vp^* \mathcal{A}_i$ is the $\Brig$-submodule of $\mathcal{B}_i$ generated by $\vp(\mathcal{A}_i)$. Clearly we have $\vp^*(\mathcal{B}_i) = \mathcal{B}_i$ for all $i$.
   \end{remark}

   \begin{lemma}
    For each $i = 1, \dots, d$, these spaces have the following properties:
    \begin{enumerate}
     \item[(a)] $\widetilde{\mathcal{A}}_i \subseteq \widetilde{\mathcal{B}}_i$.
     \item[(b)] $\widetilde{\mathcal{A}}_d \cap \widetilde{\mathcal{B}}_i = \widetilde{\mathcal{A}}_i$.
     \item[(c)] The quotient $\widetilde{\mathcal{B}}_d / (\widetilde{\mathcal{B}}_{i-1} + \widetilde{\mathcal{A}}_d)$ is annihilated by $\mathfrak{n}_{m - s_i}$.
     \item[(d)] The quotient $\widetilde{\mathcal{B}}_i / (\widetilde{\mathcal{B}}_{i-1} + \widetilde{\mathcal{A}}_i)$ is cyclic as a $\mathcal{H}(\Gamma)$-module, generated by $(1 + \pi) \vp(v_i)$, and its annihilator is $\mathfrak{n}_{m - s_i}$.
    \end{enumerate}
   \end{lemma}

   \begin{proof}
    Parts (a) and (b) are clear from the corresponding statements for the spaces $\mathcal{A}_i$ and $\mathcal{B}_i$. For part (c), we note that $\mathcal{B}_d / \mathcal{B}_{i-1}$ is isomorphic as a $(\vp, \Gamma)$-module over $\Brig$ to the tensor product 
    \[ \Brig \otimes_{\Qp} (Y_d / Y_{i-1})\]
    with $\Gamma$ acting trivially on the latter factor. By Proposition~\ref{prop:annihilator}, we have
    \[ \mathfrak{n}_k \cdot (\mathcal{B}_d / \mathcal{B}_{i-1})^{\psi = 0} = \left(\left(\tfrac{t}{\vp(\pi)}\right)^k \mathcal{B}_d / \mathcal{B}_{i-1}\right)^{\psi = 0}.\]
    Since $\mathcal{B}_d / (\mathcal{B}_{i-1} + \mathcal{A}_d)$ is annihilated by $\left(\tfrac{t}{\pi}\right)^{m - s_i}$, we deduce that $\mathfrak{n}_{m - s_i}$ annihilates $\widetilde{\mathcal{B}}_d / (\widetilde{\mathcal{B}}_{i-1} + \widetilde{\mathcal{A}}_d)$. Similarly, since $\mathcal{B}_i / (\mathcal{B}_{i-1} + \mathcal{A}_i)$ has the single elementary divisor $\left(\tfrac{t}{\pi}\right)^{m - s_i}$, we deduce that $\mathfrak{n}_{m - s_i}$ is the exact annihilator of the corresponding $\calH(\Gamma)$-module $\widetilde{\mathcal{B}}_i / (\widetilde{\mathcal{B}}_{i-1} + \widetilde{\mathcal{A}}_i)$.
   \end{proof}

   \begin{theorem}[Theorem \ref{lettertheorem:HGelem}]\label{theorem:HGelem}
    Let $W$ be any $E$-linear crystalline representation of $G_{\Qp}$ with non-negative Hodge-Tate weights $r_1 \le \dots \le r_d$. Then the $E \otimes \mathcal{H}(\Gamma)$-elementary divisors of the quotient
    \[ (\Brig)^{\psi = 0} \otimes \Dcris(W) / (\vp^* \Nrig(W))^{\psi = 0}\]
    are $[\mathfrak{n}_{r_1}; \dots; \mathfrak{n}_{r_d}]$.
   \end{theorem}

   \begin{proof}
    Let us choose an $m$ such that $V = W(-m)$ is positive. Then the Hodge-Tate weights of $V$ are $-s_1 \ge \dots \ge -s_d$, where $s_i = m - r_{d + 1 - i} \ge 0$. Suppose first that all of the eigenvalues of $\vp$ on $\Dcris(W)$, and hence on $\Dcris(V)$, are in $E$. Then there exists at least one non-critical refinement of $V$. Choosing such a refinement, we may argue as above to deduce that the $E \otimes \calH(\Gamma)$-module $M = \Dcris(W) / (\vp^* \Nrig(W))^{\psi = 0} = \mathcal{B}_d / \mathcal{A}_d$ has a filtration by $E \otimes_{\Qp} \calH(\Gamma)$-modules $M_i = \mathcal{B}_i / \mathcal{A}_i$ where $M_{i} / M_{i-1}$ is cyclic with annihilator $\mathfrak{n}_{m - s_i}$, and $\mathfrak{n}_{m - s_i}$ annihilates $M / M_{i-1}$, so the module $M$ is of the type covered by Corollary~\ref{corollary:elem-divisors}. This gives the result in this special case.

    If the Frobenius eigenvalues do not lie in $E$, let $F$ be an extension of $E$ which does contain them. Then we may consider the representation $V \otimes_{E} F$ and apply the above argument to this representation. We note that if $M$ is any $E \otimes \calH(\Gamma)$-module, then the elementary divisors of $F \otimes_E M$ as a $F \otimes \calH(\Gamma)$-module are the base extensions of the elementary divisors of $M$; by uniqueness, we have the above formula for any $E$.
   \end{proof}

   We now briefly explain how $\vp^* \Nrig(V)$ is related to the Wach module $\NN(V)$ considered in our earlier work. Note that $\calH(\Gamma)$ and $\vp(\Brig)$ are both Fr\'echet-Stein algebras in the sense of \cite{schneiderteitelbaum03} (by Theorem 5.1 of \textit{op.cit.}); hence any finite-rank free module over either one of these algebras has a canonical topology, and a submodule of such a module is finitely-generated if and only if it is closed in this topology (Corollary~3.4(ii) of \textit{op.cit.}).

  \begin{proposition}\label{prop:HGN}
   There is an isomorphism 
   \[ (\vp^* \Nrig(V))^{\psi = 0} \cong \calH(\Gamma) \otimes_{\Lambda_{\Qp}(\Gamma)} (\vp^*\NN(V))^{\psi=0}.\]
  \end{proposition}

  \begin{proof}
   We first note that $\vp^* \Nrig(V)$ is a finitely-generated $\vp(\Brig)$-submodule of $\Dcris(V) \otimes_{\Qp} (\Brig)^{\psi = 0}$. Hence it is closed in the canonical Fr\'echet topology of the latter space. It is also $\Gamma$-stable. Since the Mellin transform is a topological isomorphism between $(\Brig)^{\psi = 0}$ and $\calH(\Gamma)$, we see that $\vp^* \Nrig(V)$ is a closed $\Gamma$-stable subspace of a finite-rank free $\calH(\Gamma)$-module; hence the action of $\Gamma$ extends to a (continuous) action of $\calH(\Gamma)$. So there is a natural embedding of $\calH(\Gamma) \otimes_{\Lambda_{\Qp}(\Gamma)} (\vp^*\NN(V))^{\psi=0}$ into $(\vp^* \Nrig(V))^{\psi = 0}$.

   The image of this embedding is a $\calH(\Gamma)$-submodule, which is finitely-generated, since $(\vp^*\NN(V))^{\psi=0}$ is finitely-generated as a $\Lambda_E$-module \cite[Theorem~3.5]{leiloefflerzerbes10}. So it is closed. On the other hand, the image contains $(\vp^* \NN(V))^{\psi = 0}$. Since we evidently have $\vp^* \Nrig(V) = \vp^* \NN(V) \otimes_{\vp(\BB_{\Qp}^+)} \vp(\Brig)$, and $\vp(\BB_{\Qp}^+)$ is dense in $\vp(\Brig)$, it follows that the image of the right-hand side is a submodule of the left-hand side which is both dense and closed; hence it is everything.
  \end{proof}
  
  We recall the following result from our previous work:

  \begin{theorem}[{\cite[Lemma~3.15]{leiloefflerzerbes10}}]\label{theorem:basis}
   $(\vp^* \NN(V))^{\psi = 0}$ is a free $\Lambda_E(\Gamma)$-module of rank $d$. More specifically, for any basis $\nu_1, \dots, \nu_d$ of $\Dcris(V)$, there exists a $E \otimes \BB^+_{\Qp}$-basis $n_1, \dots, n_d$ of $\NN(V)$ such that $n_i = \nu_i \bmod \pi$ and $(1+\pi)\vp(n_1),\ldots,(1+\pi)\vp(n_d)$ form a $\Lambda_E(\Gamma)$-basis of $(\vp^* \NN(V))^{\psi = 0}$.
  \end{theorem}

  Combining this with Proposition~\ref{prop:HGN}, the following corollary is immediate:

  \begin{corollary}\label{Hmodulestructure}
   $(\vp^* \Nrig(V))^{\psi = 0}$ is a free $E\otimes\calH(\Gamma)$-module of rank $d$. More specifically, for any basis $\nu_1, \dots, \nu_d$ of $\Dcris(V)$, there exists a $E \otimes \Brig$-basis $n_1, \dots, n_d$ of $\Nrig(V)$ such that $n_i = \nu_i \bmod \pi$ and $(1+\pi)\vp(n_1),\ldots,(1+\pi)\vp(n_d)$ form a $E \otimes \calH(\Gamma)$-basis of $(\vp^* \Nrig(V))^{\psi = 0}$.
  \end{corollary}

  \begin{remark} 
   It seems reasonable to conjecture that for \emph{any} $E \otimes \Brig$-basis $m_1, \dots, m_d$ of $\Nrig(V)$, the vectors $(1 + \pi)\vp(m_i)$ are a $E \otimes \calH(\Gamma)$-basis of $(\vp^* \Nrig(V))^{\psi = 0}$, and similarly for $\NN(V)$; but we do not know a proof of this statement.
  \end{remark}


\section{The construction of Coleman maps}\label{review}

 \subsection{Coleman maps and the Perrin-Riou \texorpdfstring{$p$}{p}-adic regulator}\label{regulator}

  Let $E$ be a finite extension of $\Qp$. Let $V$ be a $d$-dimensional $E$-linear representation of $G_{\Qp}$ with non-negative Hodge-Tate weights $r_1 \le r_2 \le \dots \le r_d$. We assume that $V$ has no quotient isomorphic to the trivial representation. Let $T$ be a $\calG_{\Qp}$-stable $\calO_E$-lattice in $V$. Under these assumptions, there is a canonical isomorphism of $\Lambda_{\calO_E}(\Gamma)$-modules
  \[ h^1_{\Iw}: \NN(T)^{\psi = 1} \rTo^{\cong} H^1_{\Iw}(\Qp, T). \]
  by \cite[Theorem A.3]{berger03}. Moreover, since the Hodge-Tate weights of $V$ are non-negative, we have $\NN(T) \subseteq \vp^* \NN(T)$ by Lemma~\ref{lemma:vpstar}. Hence there is a well-defined map $1 - \vp: \NN(T) \to \vp^* \NN(T)$, which maps $\NN(T)^{\psi = 1}$ to $(\vp^* \NN(T))^{\psi = 0}$. 

  As we recalled above, \cite[Theorem 3.5]{leiloefflerzerbes10} (due to Laurent Berger) shows that for some basis $n_1, \dots, n_d$ of $\NN(T)$ as an $\EA$-module, the vectors $(1 + \pi) \vp(n_1), \dots, (1 + \pi) \vp(n_d)$ form a basis of $(\vp^*\NN(T))^{\psi=0}$ as a $\Lambda_{\calO_E}(\Gamma)$-module. This basis gives an isomorphism
  \[
  \J: (\vp^*\NN(T))^{\psi=0}\rTo^\cong \Lambda_{\calO_E}(\Gamma)^{\oplus d}
  \]
  (the \emph{Iwasawa transform}), and we define the Coleman map 
  \[ \uCol = (\uCol_i)_{i = 1}^d : \NN(T)^{\psi=1}\rightarrow \Lambda_{\calO_E}(\Gamma)^{\oplus d}\]
  as the composition $\J \circ (1 - \vp)$. 

  \begin{remark}
   This direct definition of the Coleman map is equivalent to that given in our earlier work, but applies to any representation with non-negative Hodge-Tate weights, rather than starting with a positive representation and twisting by the sum of its Hodge-Tate weights as in \cite{leiloefflerzerbes10}.
  \end{remark}


  Let $\nu_1, \dots, \nu_d$ be a basis of $\Dcris(V)$, so $(1 + \pi) \otimes \nu_1, \dots, (1 + \pi) \otimes \nu_d$ are a basis of $(\Brig)^{\psi = 0} \otimes \Dcris(V)$ as an $\calH(\Gamma)$-module; and $n_1, \dots, n_d$ be a basis of $\NN(V)$ lifting $\nu_1, \dots, \nu_d$ as in Theorem~\ref{theorem:basis}. Then there exists a unique $d\times d$ matrix $\underline{M}$ with entries in $\calH(\Gamma)$ such that
  \begin{equation}\label{eq:defnM}
   \begin{pmatrix} (1 + \pi) \vp(n_1) \\ \vdots \\ (1 + \pi) \vp(n_d) \end{pmatrix} = \underline{M} \cdot \begin{pmatrix} (1 + \pi)\otimes \nu_1 \\ \vdots \\ (1 + \pi) \otimes \nu_d \end{pmatrix}.
  \end{equation}
  In fact $\underline{M}$ is defined over $\calH(\Gamma_1)$, since the $n_i$ lie in $(1 + \pi) \vp(\NN(V)) \subseteq (1 + \pi) \vp(\Brig) \otimes \Dcris(V)$. By Theorem~\ref{theorem:HGelem}, we know that the elementary divisors of $\underline{M}$ are $\mathfrak{n}_{r_1}, \dots, \mathfrak{n}_{r_d}$.

  \begin{corollary}\label{detuM}
   Up to a unit, $\det(\underline{M})$ is equal to $\prod_{i=1}^{d}\mathfrak{n}_{r_i}$.
  \end{corollary}

  We can write the Coleman map $\uCol$ in terms of $\underline{M}$ as follows: 

  \begin{lemma}\label{colemanregulator}
   For $x\in\NN(T)^{\psi=1}$, we have
   \[
    (1-\vp)(x)= \uCol(x) \cdot \underline{M} \cdot \begin{pmatrix}(1 + \pi)\otimes \nu_{1}\\ \vdots \\ (1 + \pi)\otimes \nu_{d}\end{pmatrix}.
   \]
  \end{lemma}

  \begin{proof}
   We have by definition
   \[
   (1-\vp)x=\uCol(x) \cdot \begin{pmatrix}(1+\pi)\vp(n_1) \\ \vdots \\ (1+\pi)\vp(n_d) \end{pmatrix}.
   \]
   Therefore, we are done on combining this with \eqref{eq:defnM}.
  \end{proof}

  \begin{definition}\label{definitionregulator}
   The \emph{Perrin-Riou $p$-adic regulator} $\LL_{V}$ for $V$ is defined to be the $\Lambda_E(\Gamma)$-homomorphism
   \[
   \left(\mathfrak{M}^{-1}\otimes1\right)\circ(1-\vp)\circ \left(h^1_{\Iw,V}\right)^{-1}:H^1_{\Iw}(\Qp,V)\rTo \calH(\Gamma)\otimes\Dcris(V).
   \]
  \end{definition}

  Using the isomorphism $h^1_{\Iw, V} : \NN(V)^{\psi = 1} \to H^1_{\Iw}(\Qp,V)$, we can thus rewrite Lemma~\ref{colemanregulator} as 
  \begin{equation}\label{eq:LCol}
  \LL_{V}=\left(\uCol\circ (h^1_{\Iw,V})^{-1}\right)(z) \cdot \underline{M} \cdot \begin{pmatrix} \nu_{1}\\ \vdots \\ \nu_{d} \end{pmatrix}.
  \end{equation}


\section{Images of the Coleman maps}\label{imcol}

 Let $\eta$ be a character on $\Delta$. In this section, we study the image of $\uCol^\eta(\NN(V)^{\psi=1})$ as a subset of $\Lambda_E(\Gamma_1)^{\oplus d}$ for a crystalline representation $V$ of dimension $d$ with non-negative Hodge-Tate weights. We then consider the projection of this image, giving a description of $\image(\uCol_i^\eta)$ for $i=1,\ldots,d$.


 \subsection{Preliminary results on \texorpdfstring{$\Lambda_E(\Gamma_1)$}{LambdaE(Gamma1)}-modules}

  Recall that we identify $\Lambda_E(\Gamma_1)$ with the power series ring $E\otimes\calO_E[[X]]$ by identifying $\gamma-1$ with $X$. Therefore, if $F\in\Lambda_E(\Gamma_1)$ and $x$ is an element of the maximal ideal of $E$, $F|_{X=x}\in E$.

  \begin{lemma}\label{base}
   Let $V$ be an $E$-subspace of $E^d$ with codimension $n$. For a fixed element $x$ of the maximal ideal of $E$, we define the $\Lambda_E(\Gamma_1)$-module
   \[
   S=\left\{(F_1,\ldots,F_d)\in\Lambda_E(\Gamma_1)^{\oplus d}:(F_1(x),\ldots,F_d(x))\in V\right\}.
   \]
   Then, $S$ is free of rank $d$ over $\Lambda_E(\Gamma_1)$ and $\det(S)=(X-x)^n$.
  \end{lemma}
  \begin{proof}
   Let $v_1,\ldots,v_{d}$ be a basis of $E$ such that $\sum_{i=1}^de_iv_i\in V$ if and only if $e_i=0$ for all $i>d-n$. On multiplying elementary matrices in $GL_d(E)$ if necessary, we may assume that $S$ is of the form
   \begin{align*}
    S & =\left\{(F_1,\ldots,F_d)\in\Lambda_E(\Gamma_1)^{\oplus d}:F_{d-n+1}(x)=\cdots=F_d(x)=0\right\}\\
      & =\Lambda_E(\Gamma_1)^{\oplus(d-n)}\oplus\big((X-x)\Lambda_E(\Gamma_1)\big)^{\oplus n},
   \end{align*}
   so we are done.
  \end{proof}

  \begin{proposition}\label{general}
   Let $I=\{x_0,\ldots,x_m\}$ be a subset of the maximal ideal of $E$. For each $i$, let $V_i$ be an $E$-subspace of $E^{\oplus d}$ with codimension $n_i$. Define
   \[
   S=\left\{(F_1,\ldots,F_d)\in\Lambda_E(\Gamma_1)^{\oplus d}:\big(F_1(x_i),\ldots,F_d(x_i)\big)\in V_i, i=0,\ldots,m\right\},
   \]
   then $S$ is free of rank $d$ over $\Lambda_E(\Gamma_1)$, and $\det(S)=\prod_{i=0}^{m}(X-x_i)^{n_i}$.
  \end{proposition}
  \begin{proof}
   We prove the result by induction on $m$. The case $m=0$ is just Lemma~\ref{base}.

   Assume that $m>0$ and let 
   \[
   S'=\left\{(F_1,\ldots,F_d)\in\Lambda_E(\Gamma_1)^{\oplus d}:(F_1(x_i),\ldots,F_d(x_i))\in V_i, i=0,\ldots,m-1\right\}.
   \]
   By induction, $S'$ is free of rank $d$ over $\Lambda_E(\Gamma_1)$ and $\det(S')=\prod_{i=0}^{m-1}(X-x_i)^{n_i}$. Let $F^{(i)}=\left(F_1^{(i)},\ldots,F_d^{(i)}\right)$, $i=1,\ldots,d$, be a $\Lambda_E(\Gamma_1)$-basis of $S'$. Write $\calF_m$ for the $d\times d$ matrix with entries $F_j^{(i)}(x_m)$. As $X-x_m$ does not divide $\det(F_j^{(i)})$, we have $\calF_m\in GL_d(E)$. Define
   \[
   S''=\left\{(G_1,\ldots,G_d)\in\Lambda_E(\Gamma_1)^{\oplus d}:(G_1(x_m),\ldots,G_d(x_m))\in V_m\calF_m^{-1}\right\}.
   \]
   By Lemma~\ref{base}, $S''$ is free of rank $d$ over $\Lambda_E(\Gamma_1)$ and $\det(S'')=(X-x_m)^{n_m}$. Say, $(G_1^{(k)},\ldots, G_d^{(k)})$, $k=1,\ldots,d$, is a basis. 

   For $(G_1,\ldots, G_d)\in \Lambda_E(\Gamma_1)^{\oplus d}$, we have $\sum_{i=1}^dG_iF^{(i)}\in S'\subset S$ by definition. It is easy to check that $\sum_{i=1}^dG_iF^{(i)}\in S$ if and only if $(G_1,\ldots, G_d)\in S''$. Therefore, a basis for $S$ is given by the row vectors of $(G_i^{(k)})(F_j^{(i)})$ and $\det(S)=\det(S')\det(S'')$. Hence, we are done.
  \end{proof}

  \begin{lemma}\label{projection}
   If $S$ is a $\Lambda_E(\Gamma_1)$-module as described in the statement of Proposition~\ref{general}, then the image of a projection from $S$ into $\Lambda_E(\Gamma_1)$ is of the form $\prod_{i\in J}(X-x_i)\Lambda_E(\Gamma_1)$ where $J$ is some subset of $\{0,\ldots,m\}$.
  \end{lemma}
  \begin{proof}
   We consider the first projection $\pr_1:(F_1,\ldots,F_d)\mapsto F_1$. Let 
   \[
   J=\{i\in[0,m]:(e_1,\ldots,e_d)\in V_i\Rightarrow e_1=0\}.
   \]
   It is clear that $\image(\pr_1)\subset\prod_{i\in J}(X-x_i)\Lambda_E(\Gamma_1)$. It remains to show that $\prod_{i\in J}(X-x_i)\in\image(\pr_1)$.

   By definition, for each $i\notin J$, there exist $e_k^{(i)}\in E$, $k=2,\ldots,d$, such that 
   \[
   \left(\prod_{j\in J}(x_i-x_j),e_2^{(i)},\ldots,e_d^{(i)}\right)\in V_i.
   \]
   Similarly, take any $(0,e_2^{(i)},\ldots,e_d^{(i)})\in V_i$ for $i\in J$. There exist polynomials $F_k$ over $E$ such that $F_k(x_i)=e_j^{(i)}$ for $k=2,\ldots,d$ and $i=0,\ldots,m$. It is then clear that 
   \[
   \left(\prod_{i\in J}(X-x_i),F_2,\ldots,F_d\right)\in S.
   \] 
   Hence we are done.
  \end{proof}


 \subsection{On the image of the Perrin-Riou \texorpdfstring{$p$}{p}-adic regulator}

  
  %
  
  Let $V$ be a $d$-dimensional $E$-linear crystalline representation of $\calG_{\Qp}$ with non-negative Hodge-Tate weights $r_1 \le \dots \le r_d$. 

  \begin{definition}\label{def:ni}
   For an integer $i \ge 0$, we write 
   \[ n_i=\dim_E\Fil^{-i}\Dcris(V) = \# \{j : r_j \le i\}.\]
  \end{definition}
  
  We make the following assumption:

  \begin{assumption}
   The eigenvalues of $\vp$ on $\Dcris(V)$ are not integral powers of $p$.
  \end{assumption}
  

  Recall that we have the Perrin-Riou exponential map (c.f. \cite{perrinriou94})
  \[
  \Omega_{V,r_d}:(\Brig)^{\psi=0}\otimes\Dcris(V)\rightarrow\calH(\Gamma)\otimes H^1_{\Iw}(\Qp,V).
  \]
  The Perrin-Riou $p$-adic regulator is related to $\Omega_{V,r_d}$ via the following equation.

  \begin{theorem}\label{bergerpr}
   As maps on $H^1_{\Iw}(\Qp,V)$, we have
   \[
   \LL_{V}=\left(\mathfrak{M}^{-1}\otimes1\right)\left(\prod_{i=0}^{r_d-1}\ell_i\right)\left(\Omega_{V,r_d}\right)^{-1}.
   \]
  \end{theorem}
  \begin{proof}
  By definition, this is the same as saying
  \[
   (1-\vp)\circ\left(h^1_{\Iw,V}\right)^{-1}=\left(\prod_{i=0}^{r_d-1}\ell_i\right)\left(\Omega_{V,r_d}\right)^{-1},
  \]
  which is just a rewrite of \cite[Theorem~II.13]{berger03}.  
  \end{proof}

  \begin{corollary}\label{padicdet}
   We have
   \[
   \det(\calL_{V})=\prod_{i=0}^{r_d-1}(\ell_i)^{d-n_i}.
   \]
  \end{corollary}

  \begin{proof}
   The $\delta(V)$-conjecture (see~\cite[Conjecture 3.4.7]{perrinriou94}) predicts that 
   \[
   \det(\Omega_{V,r_d})=\prod_{i \le r_d-1}(\ell_i)^{n_i}.
   \]
   As pointed out in \cite[Proposition 3.6.7]{perrinriou94}, this conjecture is a consequence of Perrin-Riou's explicit reciprocity law (Conjecture (R\'ec) in \emph{op.cit.}) which is proved in \cite[Th\'eor\`eme IX.4.5]{colmez98}.
   Therefore, Theorem~\ref{bergerpr} implies that
   \[
    \det(\calL_{V})=\left(\prod_{i=0}^{r_d-1}(\ell_i)^d\right)\left(\prod_{i\le r_d-1}(\ell_i)^{-n_i}\right),
   \]
   which finishes the proof, since $n_i = 0$ for $i < 0$.
  \end{proof}

  Let $z\in H^1_{\Iw}(\Qp,V)$. Then $\LL_V(z)\in\calH(\Gamma)\otimes_{\Qp}\Dcris(V)$, so we can apply to $\LL_V(z)$ any character on $\Gamma$ to obtain an element in $\Dcris(V)$. The following proposition studies elements obtained in this way when we choose characters of a specific kind. Recall that we denote by $\chi$ the cyclotomic character, and by $\chi_0$ the restriction of $\chi$ to $\Delta$.

  \begin{proposition}
   Let $z\in H^1_{\Iw}(\Qp,V)$. Then for any integer $0\le i\le r_d-1$ and any Dirichlet character $\delta$ of conductor $p^n > 1$, we have
   \begin{align}
    &(1-\vp)^{-1}\left(1-p^{-1}\vp^{-1}\right)\chi^i(\LL_{V}(z)\otimes t^{i}e_{-i})\in\Fil^{0}\Dcris(V(-i));\label{1strelation}\\
    &\vp^{-n}\left(\chi^i\delta(\LL_{V}(z)\otimes t^{i}e_{-i})\right)\in\Qpn\otimes\Fil^{0}\Dcris(V(-i)).\label{2ndrelation}
   \end{align}
  \end{proposition}
  \begin{proof}
   We write $[\hspace{1ex},\hspace{1ex}]$ for the pairing 
   \[
   \Dcris(V(-i))\times\Dcris(V^*(1+i))\rTo\Dcris(E(1))=E\cdot t^{-1}e_1.
   \]
   The orthogonal complement of $\Fil^{0}\Dcris(V(-i))$ under $[\hspace{1ex},\hspace{1ex}]$ is $\Fil^0(V^*(1+i))$. Let $x\in\Fil^0\Dcris(V^*(1+i))$ and $x'=(1-\vp)(1-p^{-1}\vp^{-1})^{-1}x$, and write $x'_{-i}$ for $x'\otimes t^ie_{-i}$. Then 
   \begin{align*}
    \big[(1-\vp)^{-1}\left(1-p^{-1}\vp^{-1}\right)\chi^i(\LL_{V}(z)\otimes t^{i}e_{-i}),x\big] & =\big[\chi^i(\LL_{V}(z)\otimes t^{i}e_{-i}),x'\big] \\
   & =\chi^i[\LL_{V}(z),x'_{-i}],
   \end{align*}
   where the first equality follows from the observation that $1-\vp$ and $1-p^{-1}\vp^{-1}$ are adjoint to each other under the pairing $[\hspace{1ex},\hspace{1ex}]$. 

   We extend $[\hspace{1ex},\hspace{1ex}]$ to a pairing on
   \[
   \calH(\Gamma)\otimes_{\Qp}\Dcris(V)\times\calH(\Gamma)\otimes_{\Qp}\Dcris(V^*(1))\rTo E\otimes_{\Qp}\calH(\Gamma)
   \]
   in the natural way. By Perrin-Riou's explicit reciprocity law (c.f. \cite[Th\'eor\`eme IX.4.5]{colmez98}) and Theorem~\ref{bergerpr}, we have
   \begin{equation}\label{1rec}
   \left[\LL_{V}(z),x'_{-i}\right]=(-1)^{r_d-1}\left\langle\left(\prod_{j=0}^{r_d-1}\ell_j\right)z,\Omega_{V^*(1),1-r_d}((1+\pi)\otimes x'_{-i})\right\rangle
   \end{equation}
   where $\langle\hspace{1ex},\hspace{1ex}\rangle$ denotes the pairing
   \[
   \left(\calH(\Gamma)\otimes H^1_{\Iw}(\Qp,V)\right)\times\left(\calH(\Gamma)\otimes H^1_{\Iw}(\Qp,V^*(1))\right)\rTo E\otimes\calH(\Gamma)
   \]
   as defined in \cite[\S~3.6]{perrinriou94}. By \cite[Lemme~3.6.1(i)]{perrinriou94}, the right-hand side of \eqref{1rec} in fact equals
   \begin{equation}\label{2rec}
    \left\langle z,\left(\prod_{j=0}^{r_d-1}\ell_{-j}\right)\Omega_{V^*(1),1-r_d}((1+\pi)\otimes x'_{-i})\right\rangle=\left\langle z,\Omega_{V^*(1),1}((1+\pi)\otimes x'_{-i})\right\rangle.
   \end{equation}

   By an abuse of notation, we let $\Tw$ denote the twist map on the $H^1_{\Iw}$'s as well as the map on $\calH(\Gamma)$ that sends any $g\in\Gamma$ to $\chi(g)g$. We have
   \[
   \langle\Tw^{-i}(x),\Tw^{i}(y)\rangle=\Tw^{i}\langle x,y\rangle
   \]
   for any $x$ and $y$ by \cite[Lemme~3.6.1(ii)]{perrinriou94}. Therefore, by combining \eqref{1rec} and \eqref{2rec}, $\chi^i[\LL_{V}(z),x'_{-i}]$ is equal to the projection of 
   \[
   \left\langle\Tw^{-i}(z),\Tw^i\left(\Omega_{V^*(1),1}((1+\pi)\otimes x'_{-i})\right)\right\rangle
   \]
   into $E$. The projection of $\Tw^i\left(\Omega_{V^*(1),1}((1+\pi)\otimes x'_{-i})\right)$ into $H^1(\Qp,V^*(1+i))$ is equal to a scalar multiple of
   \[
   \exp_{\Qp,V^*(1+i)}\left((1-p^{-1}\vp^{-1})(1-\vp)^{-1}(x')\right)
   \]
   (see for example \cite[Proposition~3.19]{leiloefflerzerbes10}). But
   \[
   (1-p^{-1}\vp^{-1})(1-\vp)^{-1}(x')=x\in \Fil^0\Dcris(V^*(1+i))
   \]
   by definition. Therefore, as $\exp_{\Qp,V^*(1+i)}$ factors through $\Fil^0\Dcris(V^*(1+i))$ by construction, it follows that 
   \[\exp_{\Qp,V^*(1+i)}\left((1-\vp)^{-1}(1-p^{-1}\vp^{-1})(x')\right)=0\] 
   and hence that 
   \[
   \big[(1-\vp)^{-1}\left(1-p^{-1}\vp^{-1}\right)\Tw^i(\LL_{V}(z)\otimes t^ie_{-i}),x\big]=\chi^i\big[\LL_{V}(z),x'_{-i}\big]=0.
   \]
   This implies \eqref{1strelation}, and \eqref{2ndrelation} can be proved similarly.
  \end{proof}

  For any character $\eta$ of $\Delta$ and an integer $0 \le i \le r_d-1$, define
  \[ V_{i,\eta}=
     \begin{cases}
     (1 - p^i \vp)(1-p^{-1-i}\vp^{-1})^{-1}\Fil^{-i}\Dcris(V) & \text{if $\chi_0^i=\eta$} \\
     \vp\big(\Fil^{-i}\Dcris(V)\big) & \text{otherwise}
     \end{cases} 
  \]
  Note that $V_{i,\eta}$ is a subspace of $\Dcris(V)$ of the same dimension as $\Fil^{-i}\Dcris(V)$.

  \begin{corollary}\label{twistrelations}
   If $\eta$ is a character on $\Delta$, then 
   \[
   \big\{\chi^i\chi_0^{-i}\eta(e_\eta\LL_{V}(z)):z\in H^1_{\Iw}(\Qp,V)\big\}\subset V_{i,\eta}.
   \]
  \end{corollary}
  \begin{proof}
   Note that $\Fil^{-i}\Dcris(V)= \Fil^0\Dcris(V(-i))\otimes t^{-i}e_i$. Therefore, if $\chi_0^i=\eta$, the result follows from \eqref{1strelation} and the fact that $\vp(t^ie_{-i})=p^it^ie_{-i}$. Assume otherwise. Since $\chi^i\chi_0^{-i}\eta|_{\Delta}=\eta$, we have $\chi^i\chi_0^{-i}\eta(e_\eta\LL_{V}(z))=\chi^i\chi_0^{-i}\eta(\LL_{V}(z))$. Hence, \eqref{2ndrelation} implies that 
   \[
   \vp^{-1}\left(\chi^i\chi_0^{-i}\eta(e_\eta\LL_{V}(z)\otimes t^ie_{-i})\right)\in \Qp(\mu_p)\otimes\Fil^{0}\Dcris(V(-i)).
   \]
   But $\chi^i\chi_0^{-i}\eta(e_\eta\LL_{V}(z)\otimes t^ie_{-i})=\LL_{V}(z)^\eta|_{X=\chi(\gamma)^i-1}\otimes t^ie_{-i}$ in fact lies inside $\Dcris(V(-i))$. Hence, 
   \[
   \vp^{-1}\left(\chi^i\chi_0^{-i}\eta(e_\eta\LL_{V}(z)\otimes t^ie_{-i})\right)\in \Fil^{0}\Dcris(V(-i))=\Fil^{-i}\Dcris(V)\otimes t^ie_{-i}
   \]
   and we are done on applying $\vp$ to both sides.
  \end{proof}

  \begin{corollary}\label{relations}
   If $\eta$ is a character on $\Delta$, then
   \[
   \big\{\calL_{V}(z)^\eta|_{X=\chi(\gamma)^i-1}: z\in H^1_{\Iw}(\Qp,V)\big\}\subset V_{i,\eta}.
   \]
  \end{corollary}
  \begin{proof}
   This is immediate from Corollary~\ref{twistrelations} as $\calL_{V}(z)^\eta|_{X=\chi(\gamma)^i-1}=\chi^i\chi_0^{-i}\eta(e_\eta\LL_{V}(z))$.
  \end{proof}


 \subsection{Images of the Coleman maps}

  We now fix a character $\eta:\Delta\rTo\Zp^\times$.  Let $\nu_1, \dots, \nu_d$ be a basis of $\Dcris(V)$ and $n_1, \dots, n_d$ a basis of $\NN(V)$ lifting $\nu_1, \dots, \nu_d$ as in Theorem~\ref{theorem:basis}. We consider the image of the Coleman map defined with respect to this basis as in Section~\ref{review}. 
  
  \begin{proposition}\label{colemanrelations}
   The image of the map
   \[
   \uCol^\eta:\NN(V)^{\psi=1}\rTo\Lambda_E(\Gamma_1)^{\oplus d}
   \]
   lies inside a $\Lambda_E(\Gamma_1)$-submodule $S$ as described in the statement of Lemma~\ref{general} with $I=\{x_{i}=\chi(\gamma)^{i}-1:0\leq i\le r_d-1\}$ and $V_i=V_{i,\eta}$, which is an $E$-vector space of the same (co-)dimension as $\Fil^{-i}\Dcris(V)$.
  \end{proposition}

  \begin{proof}
   Recall from \eqref{eq:LCol} that
   \[
   \LL_{V}=\left(\uCol\circ h^1_{\Iw,V}\right)\underline{M}\begin{pmatrix}\nu_{1}\\\vdots\\\nu_{d}\end{pmatrix}
   \]
   where $\underline{M}$ is as defined in \eqref{eq:defnM}. Note that $\underline{M}^\eta=\underline{M}$ for any character $\eta$ of $\Delta$, since $\underline{M}$ is defined over $\calH(\Gamma_1)$. Moreover, Corollary~\ref{detuM} implies that $X-\chi(\gamma)^i+1$ does not divide $\det(\underline{M})$, so $\underline{M}|_{X=\chi(\gamma)^i-1}\in GL_d(E)$. Therefore, we are done by Corollary~\ref{relations}.
  \end{proof}

  \begin{theorem}\label{equalitycoleman}
   Equality holds in Proposition~\ref{colemanrelations}.
  \end{theorem}

  \begin{proof}
   Write $S$ for the basis matrix of the $\Lambda_E(\Gamma_1)$-submodule of $\Lambda_E(\Gamma_1)^{\oplus d}$ described in the statement of Proposition~\ref{colemanrelations}. Then, Proposition~\ref{general} says that
   \[
   \det(S)=\prod_{i=0}^{r_d-1}(X-\chi(\gamma)^i+1)^{d-n_i}.
   \]

   But 
   \[
   \det(\underline{M})=\prod_{j=1}^{d}\left(\prod_{i=0}^{r_j-1}\frac{\ell_i}{X-\chi(\gamma)^i+1}\right)=\prod_{i=0}^{r_d-1}\left(\frac{\ell_i}{X-\chi(\gamma)^i+1}\right)^{d - n_i},
   \]
   since $n_i = \#\{ j : r_j \le i\}$, as noted above. Hence, Corollary~\ref{padicdet} implies that
   \[
   \det(\calL_{V})=\det(\underline{M})\det(S)
   \]
   and we are done.
  \end{proof}

  We can summarize the above results via the following short exact sequence:

  \begin{corollary}\label{exactsequence}
   Suppose that no eigenvalue of $\vp$ on $\Dcris(V)$ lies in $p^\ZZ$. Then for each character $\eta$ of $\Delta$, there is a short exact sequence of $\calH(\Gamma_1)$-modules
   \[ 0 \rTo \NN(V)^{\psi = 1, \eta} \rTo^{1-\vp} (\vp^* \NN(V))^{\psi = 0, \eta} \rTo^{A_\eta} \bigoplus_{i = 0}^{r_d - 1} (\Dcris(V) / V_{i, \eta})(\chi^i \chi_0^{-i} \eta) \rTo 0.\]
   Here the map $A_{\eta} = \bigoplus_i (1 \otimes A_{\eta, i})$, where $A_{\eta, i}$ is the natural reduction map $\calH(\Gamma) \to \Qp(\chi^i \chi_0^{-i} \eta)$ obtained by quotienting out by the ideal $(X + 1 - \chi(\gamma)^i) \cdot e_{\eta}$.
  \end{corollary}
  
  \begin{remark}
   The short exact sequence in Corollary~\ref{exactsequence} can be seen as an analogue of Perrin-Riou's exact sequence (see \cite[\S 2.2]{perrinriou94})
   \begin{multline*}
    0\rTo \bigoplus_{i=0}^{r_d} t^i \Dcris(V)^{\vp = p^{-i}} \rTo \big(\Brig\otimes\Dcris(V)\big)^{\psi=1}\rTo^{\vp-1} (\Brig)^{\psi=0}\otimes\Dcris(V)\\
    \rTo\bigoplus_{i=0}^{r_d}\left(\frac{\Dcris(V)}{1-p^i\vp}\right)(i)\rightarrow 0.
   \end{multline*}
   In particular, the injectivity of the first map in our sequence follows from Perrin-Riou's sequence, since our assumption on $V$ implies that the first term of Perrin-Riou's sequence vanishes.
  \end{remark}

  We can now prove Theorem~\ref{lettertheorem:images}.

  \begin{corollary}\label{corollary:projectioncoleman}
   For $i=1,\ldots,d$, we have
   \[\image(\uCol_i^\eta)=\prod_{j\in I_i^\eta}(X-\chi(\gamma)^j+1)\Lambda_E(\Gamma_1)\]
   for some $I_i^\eta\subset\{0,\ldots,r_d-1\}$.
  \end{corollary}
  \begin{proof}
   This follows immediately from Lemma~\ref{projection}.
  \end{proof}


  We can also use this argument to determine the elementary divisors of the cokernel of the map $\LL_V$, refining the result of Proposition~\ref{padicdet}.

  \begin{theorem}\label{theorem:LVelem}
   The elementary divisors of the $\calH(\Gamma)$-module quotient
   \[ \frac{\calH(\Gamma) \otimes_{\Qp} \Dcris(V)}{\calH(\Gamma)\otimes_{\Lambda_{\Qp}(\Gamma)}{\image}(\mathcal{L}_V)} \]
   are $[\lambda_{r_1}; \dots; \lambda_{r_d}]$, where $\lambda_k = \ell_0 \ell_1 \dots \ell_{k-1}$.
  \end{theorem}
  \begin{proof}
   We know that the matrix of $\mathcal{L}_V$ is equal to $M \cdot S$, where $M$ and $S$ have elementary divisors that are coprime. Hence the elementary divisors of the product matrix are the products of the elementary divisors, which gives the above formula.
  \end{proof}


\section{The Coleman maps for modular forms}\label{imageMF}

 In this section, we fix a modular form $f$ as in Section~\ref{sect:modularforms}. We pick bases $n_1,n_2$ of $\NN(T_{\bar{f}})$ and $\bar{\nu}_1,\bar{\nu}_2$ of $\Dcris(V_{\bar{f}})$ as in \cite[Section~3.3]{leiloefflerzerbes10}. Write $V=V_{\bar{f}}(k-1)$, it has Hodge-Tate weights are $0$ and $k-1$. We consider the Coleman maps $\uCol_1^\eta$ and $\uCol_2^\eta$ defined on $\NN(V)^{\psi=1}$ where $\eta$ is a fixed character on $\Delta$. As a special case for Theorem~\ref{equalitycoleman} and Corollary~\ref{corollary:projectioncoleman}, we have the following result.

 \begin{proposition}
  There exist $1$-dimensional $E$-subspaces $V_i$ of $E^2$ for $0\leq i<k-1$ such that
  \[
   \image(\uCol^\eta)=\{(F,G)\in\Lambda_E(\Gamma_1):\left(F(\chi^i(\gamma)-1),G(\chi^i(\gamma)-1)\right)\in V_i\}.
  \]
  Moreover, for $i=1,2$, we have 
  \[\image(\uCol_i^\eta)=\prod_{j\in I_i^\eta}(X-\chi^j(\gamma)+1)\Lambda_E(\Gamma_1)\] 
  for some $I_i^\eta\subset\{0,\ldots,k-2\}$ with $I_1^\eta$ and $I_2^\eta$ disjoint.
 \end{proposition}
 \begin{proof}
  For $0\le j\le k-2$, $\Fil^{-j}\Dcris(V)$ is of dimension $1$ over $E$. Hence the first part of the proposition by Theorem~\ref{equalitycoleman}. The second part of the proposition follows by putting
  \[
  I_1^\eta=\{i:V_i=0\oplus E\}\qquad\text{and}\qquad I_2^\eta=\{i:V_i=E\oplus0\}.
  \]
 \end{proof}
 
 \begin{remark}
  Note that the second part of the proposition is a slightly stronger version of Corollary~\ref{corollary:projectioncoleman}.
 \end{remark}

 \begin{corollary}\label{niceform}
  In particular, there exist non-zero elements $r_i\in E$ for $i\in I_3^\eta:=\{0,\ldots,k-2\}\setminus(I_1^\eta\cup I_2^\eta)$ such that 
  \[
   \image(\uCol^\eta)=\left\{(F,G)\in\Lambda_E(\Gamma_1) \quad \middle|   
   \begin{aligned}
                                                        \quad F(u^i-1)&=0 \quad \text{if $i\in I^\eta_1$} \\
                                                         \quad G(u^i-1)&=0  \quad \text{if $i\in I^\eta_2$} \\
                                                         \quad F(u^i-1)&=r_iG(u^i-1) \quad \text{if $i\in I^\eta_3$}
                                                        \end{aligned}
   \right\}
  \]  
  where $u=\chi(\gamma)$.
 \end{corollary}

 The aim of this section is to study the set above in more details.


 \subsection{Some explicit linear relations}

  Recall from \cite[proof of Proposition~3.22]{leiloefflerzerbes10} that the maps $\LL_1$ and $\LL_2$ as defined in Section~\ref{sect:modularforms} satisfy
  \[
  \LL_V(z)=-\LL_2(z)\bar{\nu}_{1,k-1}+\LL_1(z)\bar{\nu}_{2,k-1}
  \]
  for any $z\in H^1_{\Iw}(\Qp,V)$. Therefore, Corollary~\ref{twistrelations} says that $\LL_1(z)$ and $\LL_2(z)$ satifsy some linear relations when evaluated at $\chi^j\delta$ for $0\le j\le k-2$ and $\delta$ some character on $\Delta$. We now make these relations explicit. First we recall that we have:

  \begin{lemma}
   Let $j,n\ge0$ be integers and $i\in\{1,2\}$. For $z\in\HIw(\Qp,V)$, we write $z_{-j,n}$ for the image of $z$ under
   \begin{equation}\label{trivialchar}
    \HIw(\Qp,V)\rightarrow\HIw(\Qp,V(-j))\rightarrow H^1(\Qpn,V(-j))
   \end{equation}
   where the first map is the twist map $(-1)^j\Tw_j$ and the second map is the projection. Then, we have
   \begin{equation}\label{nontrivialchar}
    \chi^j(\LL_i(z))=j!\big[(1-\vp)^{-1}(1-p^{-1}\vp^{-1})\nu_{i,j+1},\exp^*_{\Qp,V(j)}(z_{-j,0})\big].
   \end{equation}
   If $\delta$ is a character of $G_n$ which does not factor through $G_{n-1}$ with $n\ge1$, then
   \[
    \chi^j\delta(\LL_i(z))=\frac{j!}{\tau(\delta^{-1})}\sum_{\sigma\in G_n}\delta^{-1}(\sigma)\big[\vp^{-n}(\nu_{i,j+1}),\exp_{\QQ_{p,1},V(j)}^*(z^{\sigma}_{-j,n})\big]
   \]
   where $\tau$ denotes the Gauss sum.
  \end{lemma}
  \begin{proof}
   See for example \cite[Lemma~3.5 and (4)]{lei09}.
  \end{proof}

  \begin{lemma}\label{zeroat2}
   If $0\le j\le k-2$ and $\delta$ is a non-trivial character on $\Delta$, then $\chi^j\delta(\LL_2(z))=0$.
  \end{lemma}
  \begin{proof}
   On putting $n=1$ in \eqref{nontrivialchar}, we have
   \[
    \chi^j\delta(\LL_2(z))=\frac{j!}{\tau(\delta^{-1})}\sum_{\sigma\in\Delta}\delta^{-1}(\sigma)\big[\vp^{-1}(\nu_{2,j+1}),\exp_{\QQ_{p,1},V(j)}^*(z^{\sigma}_{-j,1})\big].
   \]
   But $\nu_2=p^{1-k}\vp(\nu_1)$, so 
   \[
    \vp^{-1}(\nu_{2,j+1})\in E\cdot\nu_{1,j+1}=\Fil^0\Dcris(V_f(j+1)).
   \]
   Therefore, we have
   \[
    [\vp^{-1}(\nu_{2,j+1}),\exp_{\QQ_{p,1},V(j)}^*(z^{\sigma})]=0
   \]
   for all $\sigma\in\Delta$ and we are done.
  \end{proof}

  \begin{lemma}\label{formula}
   If $\vp^2+a\vp+b=0$, then
   \[
    (1-\vp)^{-1}(1-p^{-1}\vp^{-1})=\frac{(1+a+pb)\vp+a(1+a+pb)+b(p-1)}{pb(1+a+b)}.
   \]
  \end{lemma}
  \begin{proof}
   We have
   \begin{align*}
    \vp^2+a\vp+b &= 0 \\
    \vp^2-1+a(\vp-1) &= -1-a-b\\
    (1-\vp)(\vp+1+a) &= 1+a+b.
   \end{align*}
   Therefore,
   \[
    (1-\vp)^{-1}=\frac{\vp+1+a}{1+a+b}.
   \]
   Similarly, we have
   \[
    \vp^{-1}=-\frac{\vp+a}{b}.
   \]
   The result then follows from explicit calculation.
  \end{proof}

  \begin{corollary}\label{relationat0}
   For $0\le j\le k-2$, we have
   \[
    (-a_p+p^{j+1}+p^{k-1-j})\chi^j(\LL_2(z))=(p-1)\chi^j(\LL_1(z)).
   \]
  \end{corollary}
  \begin{proof}
   On $\Dcris(V_{\bar{f}}(k-1-j))$, $\vp$ satisfies 
   \[\vp^2-a_pp^{-k+1+j}\vp+p^{-k+1+2j}=0,\]
   as we assume $\varepsilon(p)=1$. Let $u=1-a_pp^{-k+1+j}+p^{-k+2+2j}$, $u'=-a_pp^{-k+1+j}u+p^{-k+1+2j}(p-1)$. Then, Lemma~\ref{formula} implies that
   \[
    (u\vp+u')(-\LL_2(z)\bar{\nu}_{1,k-1-j}+\LL_1(z)\bar{\nu}_{1,k-1-j})\in\Fil^0\Dcris(V_{\bar{f}}(k-1-j))=E\nu_{1,k-1-j}.
   \]
   On writing the above expression as a linear combination of $\nu_{1,k-1-j}$ and $\nu_{2,k-1-j}$, the coefficient of the latter turns out to be
   \[
     -p^{-j}u\chi^j(\LL_2(z))+(u'+a_pp^{-k+1+j}u)\chi^r(\LL_1(z)),
   \] 
   which must be zero, hence the result.
  \end{proof}

  \begin{remark}
   The coefficient $-a_p+p^{j+1}+p^{k-1-j}$ is non-zero by the Weil's bound.
  \end{remark}

  Recall from \cite[(32)]{leiloefflerzerbes10} that we have
  \[
   \begin{pmatrix}-\LL_2&\LL_1\end{pmatrix}=\left(\uCol\circ h^1_{\Iw,V}\right)\underline{M}.
  \]
  By \cite[proof of Proposition~3.28 and Theorem~5.4]{leiloefflerzerbes10}, we have $\underline{M}|_{X=0}=A_\vp^T=\begin{pmatrix}0&p^{k-1}\\-1&a_p\end{pmatrix}$. Therefore, the relations for $j=0$ are given by
  \begin{align*}
   (-a_p+1+p^{k-2})\uCol_2(x)^\Delta|_{X=0}&=p^{k-2}(p-1)\uCol_1(x)^\Delta|_{X=0}\quad&\text{if $\eta=1$;}\\
   \uCol_2(x)^\eta|_{X=0}&=0\quad&\text{if $\eta\ne1$.}
  \end{align*}

  In particular, for the case $k=2$, we have the following analogue of \cite[Proposition~1.2]{kuriharapollack07}. 
  
  \begin{proposition}
   If $k=2$, the trivial isotypical component of the Coleman maps give a short exact sequence
   \[ 0 \rTo H^1_{\Iw}(\QQ_p,V)\rTo^{\uCol^\Delta}\Lambda_{E}(\Gamma_1)\oplus \Lambda_{E}(\Gamma_1)\rTo^{\rho}\QQ_p\rTo 0,\]
   where $\rho$ is defined by 
   \[\rho\big(g(X),h(X)\big)= (2-a_p)g(0)-(p-1)h(0).\]
  \end{proposition}


 \subsection{Integral structure of the images}

  We now describe the integral structure of $\image(\uCol_i)^\eta$. Under the notation of Corollary~\ref{niceform}, we define
  \[
   X^\eta_i=\prod_{j\in I^\eta_i}(X-\chi(\gamma)^{j}+1).
  \]
  Then, we have:

  \begin{theorem}\label{theorem:integral-images}
   For $i=1,2$, let $X_i^\eta$ be as defined above, then $\uCol_i\left(\DD(T_{\bar{f}}(k-1))^{\psi=1}\right)^\eta\subset X_i^\eta\Lambda_{\calO_E}(\Gamma_1)$. Moreover, $X_i^\eta\Lambda_{\calO_E}(\Gamma_1)/\uCol_i\left(\DD(T_{\bar{f}}(k-1))^{\psi=1}\right)^\eta$ is pseudo-null.
  \end{theorem}
  \begin{proof}
   Let
   \[
    X_k=\prod_{j=0}^{k-2}(X-\chi(\gamma)^{j}+1).
   \]
   By \cite[proof of Proposition~4.11]{leiloefflerzerbes10}, we have
   \[
    \left(\vp^{k-1}(\pi)\vp^*\NN(T_{\bar{f}}(k-1))\right)^{\psi=0}\subset(1-\vp)\NN(T_{\bar{f}}(k-1))^{\psi=1}.
   \]
   This implies that $X_k\in\image(\uCol_i)$ for $i=1,2$. Hence, we have the following inclusions:
   \[
    X_k\Lambda_{\calO_E}(\Gamma_1)\subset\uCol_i\left(\DD(T_{\bar{f}}(k-1))^{\psi=1}\right)^\eta\subset X^\eta_i\Lambda_{\calO_E}(\Gamma_1)
   \]
   for $i=1,2$. Since $X_k$ is not divisible by $\varpi_E$, the quotient
   \[
    X^\eta_i\Lambda_{\calO_E}(\Gamma_1)/X_k\Lambda_{\calO_E}(\Gamma_1)
   \]
   is a free $\calO_E$-module of finite rank. Moreover, for a coset representative say $x$, there exists an integer $n$ such that
   \[
    \varpi_E^{n}x\in\uCol_i\left(\DD(T_{\bar{f}}(k-1))^{\psi=1}\right)^\eta.
   \]
   Therefore, $\uCol_i\left(\DD(T_{\bar{f}}(k-1))^{\psi=1}\right)^\eta$ is of finite index in $X_i^\eta\Lambda_{\calO_E}(\Gamma_1)$.
  \end{proof}


 \subsection{Surjectivity via a change of basis}\label{surjectivebasechange}

  Unfortunately, we do not have an explicit description of the sets $I^\eta_i$ given by Corollary~\ref{niceform}. However, this can be resolved by choosing a different basis:

  \begin{proposition}\label{change}
   Let $S$ be a subset of $\Lambda_E(\Gamma_1)^{\oplus2}$ as defined in Corollary~\ref{niceform}. Then, there exists $A\in GL(2,E)$ such that $SA=S'$ for some $S'$ which is of the form
   \[
    \left\{(F,G)\in\Lambda_E(\Gamma_1)^{\oplus2}:F(u^i-1)=r_i'G(u^i-1), 0\le i\le k-2\right\}
   \]
   for some non-zero elements $r_i'\in E$.
  \end{proposition}
  \begin{proof}
   Let $e_1,e_2\in E$ be non-zero elements such that $e_1e_2\ne1$, then
   \[
    \begin{pmatrix}
		1 & e_2 \\
	 	e_1 & 1 \\
    \end{pmatrix}
    \in GL(2,E).
   \]
   Let $(F,G)\in \Lambda_E(\Gamma_1)^{\oplus2}$, we write $\begin{pmatrix}F' & G'\end{pmatrix}=\begin{pmatrix}F&G\end{pmatrix}A=\begin{pmatrix}F+e_1G&G+e_2F\end{pmatrix}$. We have
   \begin{align*}
    F(u^i-1)=0 &\iff F'(u^i-1)=e_1G'(u^i-1);\\
    G(u^i-1)=0 &\iff G'(u^i-1)=e_2F'(u^i-1);\\
    F(u^i-1)=r_iG(u^i-1)&\iff (e_2r_i+1)F'(u^i-1)=(e_1+r_i)G'(u^i-1).
   \end{align*}
   Therefore, we are done on choosing $e_2\ne -r_i^{-1}$ and $e_1\ne-r_i$ for all $i\in I^\eta_3$.
  \end{proof}

  \begin{remark}
   In the construction of the Coleman maps, replacing $\begin{pmatrix}n_1\\ n_2\end{pmatrix}$ by $A\begin{pmatrix}n_1\\ n_2\end{pmatrix}$ where $A\in \GL(2,E)$ is equivalent to replacing $\underline{M}$ by $A\underline{M}$.
  \end{remark}

  Therefore, on multiplying $\underline{M}$ by an appropriate matrix in $GL(2,E)$ on the left, we can make both Coleman maps surjective. But we cannot assume $\underline{M}|_{X=0}=A_{\vp}^T$ any more.

  In Proposition~\ref{change}, it is possible to choose $e_1$ and $e_2$ as elements of the maximal ideal of $\calO_E$ so that $A\in GL(2,\calO_E)$. Therefore, we have

  \begin{theorem}
   There exists a basis of $\NN(T_{\bar{f}})$ such that the corresponding Coleman maps have the following properties:
   \[
   \Lambda_{\calO_E}(\Gamma_1)/\uCol_i(\DD(T_{\bar{f}}(k-1))^{\psi=1})^{\eta}
   \]
   is pseudo-null for $i=1,2$.
  \end{theorem}


\section*{Acknowledgements}

  We are very grateful to Bernadette Perrin-Riou for giving us some unpublished notes about her $p$-adic regulator. 

\renewcommand{\MR}[1]{%
  MR \href{http://www.ams.org/mathscinet-getitem?mr=#1}{#1}.
}
\providecommand{\bysame}{\leavevmode\hbox to3em{\hrulefill}\thinspace}

\end{document}